\documentclass[10pt,twoside]{article}
\usepackage{mathrsfs}
\usepackage{amsmath}
\usepackage{amssymb}
\usepackage{fancyhdr}
\usepackage{latexsym}
\usepackage{bbding}
\usepackage{mathrsfs}
\usepackage{wasysym}

\usepackage{multicol,graphics}

\newtheorem{theorem}{Theorem}[section]
\newtheorem{lemma}[theorem]{Lemma}

\numberwithin{equation}{section}
\newenvironment{proof}[1][Proof]{\noindent\textbf{#1.} }{\hfill $\Box$}
\allowdisplaybreaks

 \makeatletter
\begin{document}
\title{{Logarithmically Improved Blow-up Criteria for a Phase Field Navier-Stokes Vesicle-Fluid Interaction Model}\footnote{This work
is partially supported by the National Natural Science Foundation of
China (11171357).}}
\author{Jihong Zhao$^{\text{a}}$\footnote{Email addresses: zhaojihong2007@yahoo.com.cn (J. Zhao); liuqao2005@163.com (Q. Liu).}\ \ and \ Qiao Liu$^{\text{b}}$\\
[0.2cm] {\small $^{\text{a}}$  College of Science, Northwest A\&F
University, Yangling,}\\
{\small   Shaanxi 712100, PR China}\\
[0.2cm] {\small $^{\text{b}}$ Department of Mathematics, Hunan Normal University, Changsha,}\\
{\small Hunan 410081, PR China}}
\date{}
\maketitle

\begin{abstract}
In this paper, we study a hydrodynamical system modeling the
deformation of vesicle membrane under external incompressible
viscous flow fields. The system is in the Eulerian formulation and
is governed by the coupling of the incompressible Navier-Stokes
equations with a phase field equation. In the three dimensional
case, we establish two logarithmically improved blow-up criteria for
local smooth solutions of this system in terms of the vorticity
field only in the homogeneous Besov spaces.


\end{abstract}

\section{Introduction}

Recently, there have been many numerical and theoretical studies on
the configurations and deformations of elastic vesicle membranes
under external flow fields
\cite{DBV97,DLL07,DLRW09,DLW04,DLW05,DLW06, LAV02,LTT11,WD08,WX12}.
The single component vesicle membranes are possibly the simplest
models for the biological cells and molecules and have widely
studied in biology, biophysics and bioengineering. Such vesicle
membranes can be formed by certain amphiphilic molecules assembled
in water to build bilayers, and exhibit a rich set of geometric
structures in various mechanical, physical and biological
environment \cite{DLW04}. In order to model and understand the
formation and dynamics of vesicle membranes and the fluid structure
interaction, one approach is to consider equations of elasticity for
the vesicle membranes and the Navier-Stokes equations for the fluid.
However, the model established in this approach is very difficult to
study and numerically simulate due to the fact that the description
for elasticity is in Lagrangian coordinate (Hooke's law) and for
fluids is in Eulerian coordinate. To overcome this difficulty, in
\cite{DLL07,DLW04}, the authors established a phase field
Navier-Stokes vesicle fluid interaction model for the vesicle shape
dynamics in flow fields via the phase field approach. In this model,
the vesicle membrane $\Gamma$ is described by a phase function
$\phi$, which is a labeling function defined on computational domain
$Q$. The function $\phi$ takes value $+1$ inside of the vesicle
membrane and $-1$ outside, with a thin transition layer of width
characterized by a small (compared to the vesicle size) positive
parameter $\varepsilon$. Obviously, the sharp transition layer of
the phase function gives a diffusive interface description of the
vesicle membrane $\Gamma$, which is recovered by the zero level set
$\{x: \phi(x)=0\}$. The advantage of introducing such a phase
function $\phi$ is to formulate the original Lagrangian description
of the membrane evolution in the Eulerian coordinates. On the other
hand, the viscous fluid is modeled by the incompressible
Navier-Stokes equations with unit density and with an external force
defined in terms of $\phi$.

\medskip

In this paper, we study the three dimensional phase field
Navier-Stokes vesicle-fluid interaction model subjecting to the
periodic boundary conditions (i.e., in torus $\mathbb{T}^{3}$),
which reads as follows:
\begin{align}\label{eq1.1}
 &{\partial_{t}}u+u\cdot\nabla u+\nabla{P}=\mu\Delta u+\frac{\delta E(\phi)}{\delta\phi}\nabla\phi\ \ &\text{in}\ \ Q\times[0,T],\\
\label{eq1.2}
 &\nabla\cdot u=0\ \ &\text{in}\ \ Q\times[0,T],\\
\label{eq1.3}
 &\partial_{t}\phi+u\cdot\nabla\phi=-\gamma\frac{\delta
 E(\phi)}{\delta\phi}\ \ &\text{in}\ \ Q\times[0,T]\;
\end{align}
with the initial conditions
\begin{equation}\label{eq1.4}
  u(x,0)=u_{0}(x) \ \text{with}\ \nabla\cdot u_{0}=0,  \ \text{and}\ \ \phi(x,0)=\phi_{0}(x)\ \
  \text{for}\ \
  x\in Q,
\end{equation}
and the boundary conditions
\begin{equation}\label{eq1.5}
  u(x+e_{i},t)=u(x,t), \quad \phi(x+e_{i},t)=\phi(x,t )\ \
  \text{for}\ \
  x\in\partial Q\times[0,T], \ i=1,2,3,
\end{equation}
where the set of vectors
$\{e_{1}=(1,0,0),e_{2}=(0,1,0),e_{3}=(0,0,1)\}$ denotes an
orthonormal basis of $\mathbb{R}^{3}$ and $Q$ is the unit square in
$\mathbb{R}^{3}$. Here $u=(u_{1},u_{2}, u_{3})\in\mathbb{R}^{3}$ and
$P=P(x,t)\in \mathbb{R}$ denote the unknown velocity vector field
and unknown pressure of the fluid, respectively. $\phi\in\mathbb{R}$
is the phase function of the vesicle membrane $\Gamma$. $E(\phi)$
denotes the physical approximation/regularization of the Helfrich
elastic bending energy for the vesicle membrane which is given by
(cf. \cite{DLL07, DLRW09,DLW05,DLW06})
\begin{equation}\label{eq1.6}
  E(\phi)=E_{\varepsilon}(\phi)+\frac{1}{2}M_{1}(A(\phi)-\alpha)^{2}+\frac{1}{2}M_{2}(B(\phi)-\beta)^{2}
\end{equation} with
\begin{equation}\label{eq1.7}
  E_{\varepsilon}(\phi)=\frac{k}{2\varepsilon}\int_{\Omega}|f(\phi)|^{2}dx
  \ \text{and}\
  f(\phi)=-\varepsilon\Delta\phi+\frac{1}{\varepsilon}(\phi^{2}-1)\phi,
\end{equation}
where $\varepsilon$ is a small (compared to the vesicle size)
positive parameter that characterizes the thickness of transition
layer of the phase function, $M_{1}$ and $M_{2}$ are two penalty
constants which are introduced in order to enforce the volume
\begin{equation}\label{eq1.8}
  A(\phi)=\int_{\Omega}\phi \;dx
\end{equation}
and the surface area
\begin{equation}\label{eq1.9}
B(\phi)=\int_{Q}\Big(\frac{\varepsilon}{2}|\nabla\phi|^{2}+\frac{1}{4\varepsilon}(\phi^{2}-1)^{2}\Big)dx
\end{equation}
of the vesicle conserved (in time), and $\alpha=A(\phi_{0})$ and
$\beta=B(\phi_{0})$ are determined by the initial value of the phase
function $\phi_{0}$.  The positive constants $\nu$, $k$, and
$\gamma$ denote, respectively, the viscosity of the fluid, the
bending modulus of the vesicle, and the mobility coefficient.
$\frac{\delta E(\phi)}{\delta\phi}$ is the so-called chemical
potential that denotes the variational derivative of $E(\phi)$ in
the variable $\phi$. Note that, if we denote
\begin{equation}\label{eq1.10}
  g(\phi)=-\Delta f(\phi)+\frac{1}{\varepsilon^{2}}(3\phi^{2}-1)f(\phi),
\end{equation}
then a direct calculation yields that the variation of the
approximate elastic energy is given by (see \cite{DLL07,DLRW09})
\begin{align}\label{eq1.11}
  \frac{\delta
  E(\phi)}{\delta\phi}&=kg(\phi)+M_{1}(A(\phi)-\alpha)+M_{2}(B(\phi)-\beta)f(\phi)\nonumber\\
  &=k\varepsilon\Delta^{2}\phi-\frac{k}{\varepsilon}\Delta(\phi^{3}-\phi)-\frac{k}{\varepsilon}(3\phi^{2}-1)\Delta\phi+
  \frac{k}{\varepsilon^{3}}(3\phi^{2}-1)(\phi^{2}-1)\phi\nonumber\\
  &\ \ \ +M_{1}(A(\phi)-\alpha)+M_{2}(B(\phi)-\beta)f(\phi).
\end{align}

The system \eqref{eq1.1}--\eqref{eq1.3} describes the dynamic
evolution of vesicle membranes immersed in an incompressible,
Newtonian fluid, using an energetic variational approach
\cite{DLL07,DLW04}  (see \cite{DLRW05,DLRW09,DLW06,OH89,WD08} for
numerical simulations and other studies). Equations \eqref{eq1.1}
and \eqref{eq1.2} are the momentum conservation and the mass
conservation equations of a viscous fluid with unit density and with
an external force caused by the phase field $\phi$. Equation
\eqref{eq1.2} is the condition of incompressibility. Equation
\eqref{eq1.3} is a relaxed transport equation of $\phi$ with
advection by the velocity field $u$. The right-hand side of
\eqref{eq1.3} is a regularization term which ensures the consistent
dissipation of energy. Roughly speaking, the system
\eqref{eq1.1}--\eqref{eq1.3} is governed by the coupling of the
hydrodynamic fluid flow and the bending elastic properties of the
vesicle membrane. The resulting membrane configuration and the flow
field reflect the competition and the coupling of the kinetic energy
and membrane elastic energies.

\medskip

For the system \eqref{eq1.1}--\eqref{eq1.3} subjecting to no-slip
boundary condition for the velocity field and Dirichlet boundary
condition for the phase function, Du, Li and Liu in \cite{DLL07}
obtained that there exists global weak solutions via the modified
Galerkin argument, and there holds basic energy inequality
\begin{equation}\label{eq1.12}
  \frac{d}{dt}\Big(\frac{1}{2}\|u(\cdot,t)\|_{L^{2}}^{2}+E(\phi(\cdot,t))\Big)+\mu\|\nabla u(\cdot,t)\|_{L^{2}}^{2}
  +\gamma\|\frac{\delta E(\phi)}{\delta\phi}\|_{L^{2}}^{2}=0,\ \ \forall\
  t>0.
\end{equation}
Moreover, the authors also proved that weak solution is unique under
an additional regularity assumption $u\in L^{8}(0,T; L^{4}(Q))$.
Recently, local in time existence and uniqueness of strong solution
to the system \eqref{eq1.1}--\eqref{eq1.3} have been established in
\cite{LTT11}, and under the assumptions that the initial data and
the quantity $(|\Omega|+\alpha)^{2}$ are sufficiently small, the
authors proved existence of almost global strong solutions. Note
that they have to restrict the working space with proper limited
regularity due to some compatibility conditions at the boundary
which is required in the fixed point strategy. Very recently, Wu and
Xu \cite{WX12} considered the system \eqref{eq1.1}--\eqref{eq1.3}
with initial conditions \eqref{eq1.4} and periodic boundary
conditions \eqref{eq1.5} to avoid troubles caused by the boundary
terms when performing integration by parts. They proved that, for
any given initial data $(u_{0}, \phi_{0})\in H^{1}_{per}(Q)\times
H^{4}_{per}(Q)$, there exists a positive time $T$ such that the
system \eqref{eq1.1}--\eqref{eq1.5} admits a unique smooth solution
$(u, \phi)$ satisfying
\begin{equation}\label{eq1.13}
\begin{cases}
  u\in C([0,T], H^{1}_{per}(Q))\cap L^{2}(0,T; H^{2}_{per}(Q))\cap H^{1}(0,T; L^{2}_{per}(Q)),\\
  \phi \in C([0,T], H^{4}_{per}(Q))\cap L^{2}(0,T; H^{6}_{per}(Q))\cap
  H^{1}(0,T; H^{2}_{per}(Q)).
\end{cases}
\end{equation}
Moreover, if the viscosity $\mu$ is assumed to be properly large,
then the system \eqref{eq1.1}--\eqref{eq1.5} admits a unique global
strong solution that is uniformly bounded in $H^{1}_{per}\times
H^{4}_{per}$ on $[0,\infty)$. However, as for the well-known
Navier-Stokes equations, an outstanding open problem is whether or
not smooth solution of \eqref{eq1.1}--\eqref{eq1.5} on $[0,T)$ will
lead to a singularity at the time $t=T$.

\medskip

For the Navier-Stokes equations, some results were obtained in early
by Prodi \cite{P59}, Serrin \cite{S62} and Giga \cite{G86}, they
proved that if
\begin{equation}\label{eq1.14}
  \int_{0}^{T}\|u(\cdot, t)\|_{L^{p}}^{q}\;dt<\infty\
  \ \  \text{with}\ \
  \frac{3}{p}+\frac{2}{q}=1,\ 3< p\leq \infty,
\end{equation}
then the smooth solution $u$ can be extended past the time $T$,
while the limit case $p=3$ was proved by Escauriaza et al.
\cite{ESS03}. In 1995, Beir\~{a}o da Veiga \cite{B95} established
similar criterion for the derivative of the solution, i.e.,
\eqref{eq1.14} can be replaced by the following condition:
\begin{equation}\label{eq1.15}
  \int_{0}^{T}\|\nabla u(\cdot, t)\|_{L^{p}}^{q}\;dt<\infty\
  \ \  \text{with}\ \
  \frac{3}{p}+\frac{2}{q}=2,\ 3< p\leq \infty.
\end{equation}
In 1984, Beale, Kato and Majda in their pioneer work \cite{BKM84}
showed that if the smooth solution $u$ blows up at the time $t=T$,
then
\begin{equation}\label{eq1.16}
  \int_{0}^{T}\|\omega(\cdot,
  t)\|_{L^{\infty}}\;dt=\infty,
\end{equation}
where $\omega=\nabla\times u$ is the vorticity of the velocity
field. Later, Kozono and Taniuchi \cite{KT00} and Konozo, Ogawa and
Taniuchi \cite{KOT02} refined the criterion \eqref{eq1.16} to
\begin{equation}\label{eq1.17}
  \int_{0}^{T}\|\omega(\cdot, t)\|_{BMO}\;dt=\infty\
  \ \text{and}\ \ \int_{0}^{T}\|\omega(\cdot, t)\|_{\dot{B}^{0}_{\infty,
  \infty}}\;dt=\infty,
\end{equation}
respectively, where $BMO$ is the space of  \textit{Bounded Mean
Oscillation} and $\dot{B}^{0}_{\infty, \infty}$ is the homogeneous
Besov spaces.  Recently, Fan et al. \cite{FJNZ11} and Guo and Gala
\cite{GG11} improved the above criteria to the following two
logarithmic type criteria:
\begin{equation}\label{eq1.18}
  \int_{0}^{T}\frac{\|w(\cdot,t)\|_{\dot{B}^{0}_{\infty,\infty}}}
  {\sqrt{1+\ln(e+\|w(\cdot,t)\|_{\dot{B}^{0}_{\infty,\infty}})}}\;dt=\infty
\end{equation}
and
\begin{equation}\label{eq1.19}
  \int_{0}^{T}\frac{\|w(\cdot,t)\|^{2}_{\dot{B}^{-1}_{\infty,\infty}}}
  {1+\ln(e+\|w(\cdot,t)\|_{\dot{B}^{-1}_{\infty,\infty}})}\;dt=\infty.
\end{equation}
When the phase function $\phi$ is considered, similar regularity
criteria as \eqref{eq1.14} and \eqref{eq1.15} for the system
\eqref{eq1.1}--\eqref{eq1.3} have been established in \cite{WX12}.
The first author of the present paper in \cite{Z12} obtained that
the Beale-Kato-Majda  criterion \eqref{eq1.16} still holds for the
system \eqref{eq1.1}--\eqref{eq1.5}.
\medskip

Due to the lack of global well-poedness theory of the system
\eqref{eq1.1}--\eqref{eq1.5}, the investigations of blow-up criteria
of local smooth solution are very important ways to understand the
properties of solutions. Motivated by the above results, the purpose
of this paper is to establish the blow-up criteria for the system
\eqref{eq1.1}--\eqref{eq1.5} in term of the norm of the homogeneous
Besov space. The main results of this paper are as follows:

\begin{theorem}\label{th1.1}
For $(u_{0}, \phi_{0})\in H^{3}_{per}(Q)\times H^{6}_{per}(Q)$ with
$\nabla\cdot u_{0}=0$. Let $T_{*}$ be the maximal existence time
such that the system \eqref{eq1.1}--\eqref{eq1.5} has a unique
smooth solution  $(u,\phi)$ on $[0,T_{*})$. If $T_{*}<\infty$, then
\begin{align}\label{eq1.20}
\int_{0}^{T_{*}}\frac{\|\omega(\cdot,t)\|_{\dot{B}^{0}_{\infty,\infty}}}{\sqrt{1+\ln(e+\|\omega(\cdot,t)\|_{\dot{B}^{0}_{\infty,\infty}})}}dt=\infty,
\end{align}
where $\omega=\nabla\times u$ is the vorticity field. In particular,
\begin{equation*}
\limsup_{t\nearrow
T_{*}}\|\omega(\cdot,t)\|_{\dot{B}^{0}_{\infty,\infty}}=\infty.
\end{equation*}
\end{theorem}

\begin{theorem}\label{th1.2}
For $(u_{0}, \phi_{0})\in H^{3}_{per}(Q)\times H^{6}_{per}(Q)$ with
$\nabla\cdot u_{0}=0$. Let $T_{*}$ be the maximal existence time
such that the system \eqref{eq1.1}--\eqref{eq1.5} has a unique
smooth solution  $(u,\phi)$ on $[0,T_{*})$. If $T_{*}<\infty$, then
\begin{align}\label{eq1.21}
\int_{0}^{T_{*}}\frac{\|\omega(\cdot,t)\|_{\dot{B}^{-1}_{\infty,\infty}}^{2}}{1+\ln(e+\|\omega(\cdot,t)\|_{\dot{B}^{-1}_{\infty,\infty}})}dt=\infty.
\end{align} In particular,
\begin{equation*}
\limsup_{t\nearrow
T_{*}}\|\omega(\cdot,t)\|_{\dot{B}^{-1}_{\infty,\infty}}=\infty.
\end{equation*}
\end{theorem}

\noindent\textbf{Remark 1.1} Theorems \ref{th1.1} and \ref{th1.2}
are still true, if we replace the vorticity $\omega$ by $\nabla u$
in \eqref{eq1.20} and \eqref{eq1.21}, due to the boundedness of
Riesz transforms in $\dot{B}^{0(per)}_{\infty,\infty}(Q)$ and
$\dot{B}^{-1(per)}_{\infty,\infty}(Q)$. For the definitions of these
spaces, see Section 2.

\noindent\textbf{Remark 1.2} Since $L^{\infty}(Q)\hookrightarrow
\dot{B}^{0(per)}_{\infty,\infty}(Q)$, the result \eqref{eq1.20}
improves the Beale-Kato-Majda blow-up criterion in [25].

\noindent\textbf{Remark 1.3} Observe that $\nabla
u\in\dot{B}^{-1(per)}_{\infty,\infty}(Q)$ is equivalent to
$u\in\dot{B}^{0(per)}_{\infty,\infty}(Q)$ and the Sobelev embedding
$L^{3}_{per}(Q)\hookrightarrow
\dot{B}^{-1(per)}_{\infty,\infty}(Q)$. Therefore, Theorem
\ref{th1.2} implies that if $T_{*}<\infty$, then
\begin{align*}
  &(i)\ \  \int_{0}^{T_{*}}\frac{\|u(\cdot,t)\|_{\dot{B}^{0}_{\infty,\infty}}^{2}}{1+\ln(e+\|u(\cdot,t)\|_{\dot{B}^{0}_{\infty,\infty}})}dt=\infty,\\
  &(ii)\ \ \int_{0}^{T_{*}}\frac{\|\nabla u(\cdot,t)\|_{L^{3}}^{2}}{1+\ln(e+\|\nabla
u(\cdot,t)\|_{L^{3}})}dt=\infty.
\end{align*}

 The rest of the paper is arranged as
follows. In Section 2, we recall the Littlewood-Paley decomposition
and definition of the homogeneous Besov spaces, and review some
crucial lemmas. In Section 3, we establish the bound of
$\|\nabla\Delta\phi\|_{L^{2}}$, which enables us to derive some
specific higher-order energy estimates. In Section 4, we present the
proof of Theorem \ref{th1.1}. Section 5 is devoted to the proof of
Theorem \ref{th1.2}.

\section{Preliminaries}

We first recall the Littlewood-Paley decomposition. Let
$\mathcal{S}(\mathbb{R}^{3})$ be the Schwartz class of rapidly
decreasing function, and $\mathcal{S}'(\mathbb{R}^{3})$ be its dual.
Given $f\in\mathcal{S}(\mathbb{R}^{3})$, the Fourier transform of
it, $\mathcal{F}(f)=\widehat{f}$, is defined by
$$
  \mathcal{F}(f)(\xi)=\widehat{f}(\xi)=\frac{1}{(2\pi)^{3/2}}\int_{\mathbb{R}^{3}}f(x)e^{-ix\cdot\xi}\;dx.
$$
For any given $g\in\mathcal{S}(\mathbb{R}^{3})$, the inverse Fourier
transform $\mathcal{F}^{-1}g=\check{g}$ is defined by
$$
  \mathcal{F}^{-1}(g)(x)=\check{g}(x)=\frac{1}{(2\pi)^{3/2}}\int_{\mathbb{R}^{3}}g(\xi)e^{ix\cdot\xi}\;d\xi.
$$ Let
$\mathcal{D}_{1}=\{\xi\in\mathbb{R}^{3},\ |\xi|\leq\frac{4}{3}\}$
and $\mathcal{D}_{2}=\{\xi\in\mathbb{R}^{3},\ \frac{3}{4}\leq
|\xi|\leq\frac{8}{3}\}$. Choose two non-negative radial functions
$\phi, \psi\in\mathcal{S}(\mathbb{R}^{3})$ supported, respectively,
in $\mathcal{D}_{1}$ and $\mathcal{D}_{2}$ such that
\begin{align*}
  \psi(\xi)+\sum_{j\geq0}\phi(2^{-j}\xi)=1, \ \ \xi\in\mathbb{R}^{3},\\
  \sum_{j\in\mathbb{Z}}\phi(2^{-j}\xi)=1, \ \
  \xi\in\mathbb{R}^{3}\backslash\{0\}.
\end{align*}
Let $h=\mathcal{F}^{-1}\phi$ and $\tilde{h}=\mathcal{F}^{-1}\psi$.
Then we define the dyadic blocks $\Delta_{j}$ and $S_{j}$ as
follows:
\begin{align*}
  \Delta_{j}f=\phi(2^{-j}D)f=2^{3j}\int_{\mathbb{R}^{3}}h(2^{j}y)f(x-y)\;dy,\\
  S_{j}f=\psi(2^{-j}D)f=2^{3j}\int_{\mathbb{R}^{3}}\tilde{h}(2^{j}y)f(x-y)\;dy,
\end{align*}
where $D=(D_1, D_2, D_3)$ and $D_j=i^{-1}\partial_{x_j}$
($i^{2}=-1$). The set $\{\Delta_{j}, S_{j}\}_{j\in\mathbb{Z}}$ is
called the Littlewood-Paley decomposition. Formally,
$\Delta_{j}=S_{j}-S_{j-1}$ is a frequency projection to the annulus
$\{|\xi|\sim 2^{j}\}$, and $S_{j}=\sum_{k\leq j-1}\Delta_{k}$ is a
frequency projection to the ball $\{|\xi|\leq 2^{j}\}$. For more
details, please refer to \cite{L02}.

\medskip

Next we recall the definition of homogeneous Besov spaces. Let
$\mathcal{S}_{h}'(\mathbb{R}^{3})$ be the space of temperate
distributions $f$ such that
\begin{equation*}
\lim_{j\rightarrow-\infty}S_{j}f=0\ \ \text{in}\ \
\mathcal{S}'(\mathbb{R}^{3}).
\end{equation*}
For
$s\in\mathbb{R}$ and $(p,q)\in[1, \infty]\times[1, \infty]$, the
homogeneous Besov space $\dot{B}^{s}_{p,q}(\mathbb{R}^{3})$ is
defined by
\begin{equation*}
  \dot{B}^{s}_{p,q}(\mathbb{R}^{3})=\big\{f\in \mathcal{S}'_{h}(\mathbb{R}^{3}):\ \
  \|f\|_{\dot{B}^{s}_{p,q}}<\infty\big\},
\end{equation*}
where
\begin{equation*}
  \|f\|_{\dot{B}^{s}_{p,q}}=
\begin{cases}
  \Big(\sum_{j\in\mathbb{Z}}2^{jsq}\|\Delta_{j}f\|_{L^{p}}^{q}\Big)^{1/q}\
  \ \text{for}\ \ 1\leq q<\infty,\\
  \sup_{j\in\mathbb{Z}}2^{js}\|\Delta_{j}f\|_{L^{p}}\ \ \ \ \ \ \ \ \ \ \text{for}\
  \ q=\infty.
\end{cases}
\end{equation*}
It is well-known that if either $s<\frac{3}{p}$ or $s=\frac{3}{p}$
and $q=1$, then
$(\dot{B}^{s}_{p,q}(\mathbb{R}^{3}),\|\cdot\|_{\dot{B}^{s}_{p,q}})$
is a Banach space. In particular, when $p=q=2$, we get the
homogeneous Sobolev space
$\dot{H}^{s}(\mathbb{R}^{3})=\dot{B}^{s}_{2,2}(\mathbb{R}^{3})$
which is endowed the equivalent norm
$\|f\|_{\dot{H}^{s}}=\|(-\Delta)^{s/2}f\|_{L^{2}}$. The notation
$H^{s}(\mathbb{R}^{3})$ is the standard inhomogeneous Sobolev spaces
which is endowed the standard norm
$\|f\|_{H^{s}}=\|(-\Delta)^{s/2}f\|_{L^{2}}+\|f\|_{L^{2}}$.

\medskip

We also need to introduce some well-established functional settings
for periodic problems: For $1\leq r\leq \infty$, we denote by
\begin{equation*}
  L^{r}_{per}(Q):=\{u\in L^{r}(\mathbb{R}^{3})\ |\  u(x+e_{i})=u(x)\}
\end{equation*}  endowed the usual norm $\|\cdot\|_{L^{r}}$. For an integer $m>0$, we denote by
\begin{equation*}
  H^{m}_{per}(Q):=\{u\in H^{m}(\mathbb{R}^{3})\ |\ u(x+e_{i})=u(x)\}
\end{equation*} endowed with the usual norm
$\|u\|_{H^{m}}$. For $s\in\mathbb{R}$ and $(p,q)\in[1,
\infty]\times[1, \infty]$, we denote by
\begin{equation*}
  \dot{B}^{s(per)}_{p,q}(Q)=
  \Big\{u\in \dot{B}^{s}_{p,q}(\mathbb{R}^{3}):\ |\  u(x+e_{i})=u(x)\Big\}
\end{equation*}
associated with the norm $\|\cdot\|_{\dot{B}^{s}_{p,q}}$.

\medskip

Before ending this section, we state some well-known inequalities.
The first one comes from \cite{KP88}:

\begin{lemma}\label{le2.1} {\em(\cite{KP88})}
For $s>1$, we have
\begin{equation}\label{eq2.1}
   \|\nabla^{s}(fg)-f\nabla^{s}g\|_{L^{p}}\leq
   C\big(\|\nabla f\|_{L^{p_{1}}}\|\nabla^{s-1}g\|_{L^{q_{1}}}+\|\nabla^{s}f\|_{L^{p_{2}}}\|g\|_{L^{q_{2}}}\big)
\end{equation}
with $1<p, q_{1}$, $p_{2}<\infty$ such that
$\frac{1}{p}=\frac{1}{p_{1}}+\frac{1}{q_{1}}=\frac{1}{p_{2}}+\frac{1}{q_{2}}$.
\end{lemma}

The second one can be found in \cite{KOT02} and the proof follows
from the Littlewood-Paley decomposition.

\begin{lemma}\label{le2.2} {\em(\cite{KOT02})}
For all $f\in H^{s-1}(\mathbb{R}^{3})$ with $s>\frac{5}{2}$, we have
\begin{equation}\label{eq2.2}
   \|f\|_{L^{\infty}}\leq
   C\big(1+\|f\|_{\dot{B}^{0}_{\infty,\infty}}\ln^{1/2}(e+\|f\|_{H^{s-1}})\big).
\end{equation}
\end{lemma}

The last one comes from \cite{MGO96}, see also \cite{GG11}.
\begin{lemma}\label{le2.3} {\em(\cite{MGO96, GG11})}
For all $f\in \dot{H}^{1}(\mathbb{R}^{3})\cap
\dot{B}^{-1}_{\infty,\infty}(\mathbb{R}^{3})$, we have
\begin{equation}\label{eq2.3}
   \|f\|_{L^{4}}\leq C
   \|f\|_{\dot{B}^{-1}_{\infty,\infty}}^{1/2}\|
   f\|_{\dot{H}^{1}}^{1/2}.
\end{equation}
\end{lemma}

\section{The bound of $\|\nabla\Delta\phi\|_{L^{2}}$}

By the basic energy estimate \eqref{eq1.12}, we can easily get the
following uniform estimates (cf. \cite{DLL07,WX12}):
\begin{equation}\label{eq3.1}
  \|u(\cdot,t)\|_{L^{2}}+\|\phi(\cdot,t)\|_{H^{2}}\leq C \ \ \text{for
  all}\ \
  t\geq0,
\end{equation}
\begin{equation}\label{eq3.2}
  \int_{0}^{+\infty}\Big(\mu\|\nabla u(\cdot,t)\|_{L^{2}}^{2}+\gamma\|\frac{\delta E}{\delta \phi}(\cdot,t)\|_{L^{2}}^{2}\Big)dt\leq
  C,
\end{equation}
where $C$ is a constant depending only on $\|u_{0}\|_{L^{2}}$,
$\|\phi_{0}\|_{H^{2}}$ and coefficients of the system except the
viscosity $\mu$.

\begin{lemma}\label{le3.1}
Assume that $(u_{0}, \phi_{0})\in H^{3}_{per}(Q)\times
H^{6}_{per}(Q)$ with $\nabla\cdot u_{0}=0$. For any smooth solution
$(u,\phi)$ to the system \eqref{eq1.1}--\eqref{eq1.5}, we have
\begin{align}\label{eq3.3}
  \sup_{0\leq t\leq T}\|\nabla\Delta\phi(\cdot,t)\|_{L^{2}}\leq C
\end{align}
for any $0<T<\infty$, where $C$ is a constant depending only on
$\|u_{0}\|_{H^{1}}$, $\|\phi_{0}\|_{H^{3}}$, $T$ and coefficients of
the system.
\end{lemma}
\begin{proof}
Taking $\Delta$ on \eqref{eq1.3}, multiplying the resultant by
$-\Delta^{2}\phi$, and integrating over $Q$, we obtain
\begin{align}\label{eq3.4}
  \frac{1}{2}\frac{d}{dt}\|\nabla\Delta\phi\|_{L^{2}}^{2}=-\int_{Q}\nabla\cdot(u\cdot\nabla\phi)\cdot\nabla\Delta^{2}\phi dx
  -\gamma\int_{Q}\nabla\frac{\delta
  E}{\delta\phi}\cdot\nabla\Delta^{2}\phi dx:=I_{1}+I_{2}.
\end{align}
For $I_{1}$, by using the interpolation inequality
$\|\nabla^{2}\phi\|_{L^{3}}^{2}\leq
C\|\nabla^{2}\phi\|_{L^{2}}\|\nabla\Delta\phi\|_{L^{2}}$, we can
infer from \eqref{eq3.1} that
\begin{align}\label{eq3.5}
  I_{1}&\leq
  \frac{k\gamma\varepsilon}{8}\|\nabla\Delta^{2}\phi\|_{L^{2}}^{2}+C\|\nabla u\cdot\nabla\phi\|_{L^{2}}^{2}
  +C\|u\cdot\nabla^{2}\phi\|_{L^{2}}^{2}\nonumber\\
  &\leq \frac{k\gamma\varepsilon}{8}\|\nabla\Delta^{2}\phi\|_{L^{2}}^{2}+C\big(\|\nabla u\|_{L^{2}}^{2}\|\nabla\phi\|_{L^{\infty}}^{2}
  +\|u\|_{L^{6}}^{2}\|\nabla^{2}\phi\|_{L^{3}}^{2}\big)\nonumber\\
  &\leq \frac{k\gamma\varepsilon}{8}\|\nabla\Delta^{2}\phi\|_{L^{2}}^{2}+C\|\nabla
  u\|_{L^{2}}^{2}(\|\nabla\Delta\phi\|_{L^{2}}^{2}+1)
  +C\|\nabla u\|_{L^{2}}^{2}\|\nabla^{2}\phi\|_{L^{2}}\|\nabla\Delta\phi\|_{L^{2}}\nonumber\\
  &\leq \frac{k\gamma\varepsilon}{8}\|\nabla\Delta^{2}\phi\|_{L^{2}}^{2}+C\|\nabla
  u\|_{L^{2}}^{2}(\|\nabla\Delta\phi\|_{L^{2}}^{2}+1).
\end{align}
For $I_{2}$, since $A(\phi)$ and $B(\phi)$ are functions depending
only on time, by \eqref{eq1.11}, we obtain
\begin{align}\label{eq3.6}
  I_{2}&=-\gamma\int_{Q}\nabla\Big[kg(\phi)+M_{1}(A(\phi)-\alpha)+M_{2}(B(\phi)-\beta)f(\phi)\Big]\cdot\nabla\Delta^{2}\phi
  dx\nonumber\\
  &=k\gamma\int_{Q}\nabla\Delta f(\phi)\cdot\nabla\Delta^{2}\phi
  dx-\frac{k\gamma}{\varepsilon^{2}}\int_{Q}\nabla[(3\phi^{2}-1)f(\phi)]\cdot\nabla\Delta^{2}\phi dx\nonumber\\
  &\ \ \
  -M_{2}\gamma(B(\phi)-\beta)\int_{Q}\nabla f(\phi)\cdot\nabla\Delta^{2}\phi
  dx\nonumber\\&:=I_{21}+I_{22}+I_{23}.
\end{align}
Note that
$f(\phi)=-\varepsilon\Delta\phi+\frac{1}{\varepsilon}(\phi^{2}-1)\phi$,
by \eqref{eq3.1}, we can estimate $I_{2i}$ $(i=1,2,3)$ as follows:
\begin{align}\label{eq3.7}
  I_{21}&=-k\varepsilon\gamma\|\nabla\Delta^{2}\phi\|_{L^{2}}^{2}+\frac{k\gamma}{\varepsilon}\int_{Q}\nabla\Delta(\phi^{3}-\phi)\cdot\nabla\Delta^{2}\phi
  dx\nonumber\\
  &\leq
  -\frac{7k\varepsilon\gamma}{8}\|\nabla\Delta^{2}\phi\|_{L^{2}}^{2}+C\|\nabla\Delta(\phi^{3}-\phi)\|_{L^{2}}^{2}\nonumber\\
  &\leq
  -\frac{7k\varepsilon\gamma}{8}\|\nabla\Delta^{2}\phi\|_{L^{2}}^{2}+C\big(\|\phi\|_{L^{\infty}}^{4}\|\nabla\Delta\phi\|_{L^{2}}^{2}
  +\|\phi\|_{L^{\infty}}^{2}\|\nabla\phi\|_{L^{6}}^{2}\|\Delta\phi\|_{L^{3}}^{2}\nonumber\\
  &\ \ \
  +\|\nabla\phi\|_{L^{6}}^{6}+\|\nabla\Delta\phi\|_{L^{2}}^{2}\big)\nonumber\\
  &\leq
  -\frac{7k\varepsilon\gamma}{8}\|\nabla\Delta^{2}\phi\|_{L^{2}}^{2}+C\big(\|\nabla\Delta\phi\|_{L^{2}}^{2}+1\big);
\end{align}
\begin{align}\label{eq3.8}
  I_{22}&=-\frac{6k\gamma}{\varepsilon^{2}}\int_{Q}\phi\nabla\phi f(\phi)\cdot\nabla\Delta^{2}\phi dx
  -\frac{k\gamma}{\varepsilon^{2}}\int_{Q}(3\phi^{2}-1)\nabla f(\phi)\cdot\nabla\Delta^{2}\phi dx\nonumber\\
  &\leq
  \frac{k\varepsilon\gamma}{8}\|\nabla\Delta^{2}\phi\|_{L^{2}}^{2}+C\Big(\|\phi\nabla\phi f(\phi)\|_{L^{2}}^{2}
  +\|(3\phi^{2}-1)\nabla f(\phi)\|_{L^{2}}^{2}\Big)\nonumber\\
  &\leq
  \frac{k\varepsilon\gamma}{8}\|\nabla\Delta^{2}\phi\|_{L^{2}}^{2}+
  C\Big(\|\phi\nabla\phi\Delta\phi\|_{L^{2}}^{2}+\|\phi^{2}(\phi^{2}-1)\nabla\phi\|_{L^{2}}^{2}\nonumber\\
  &\ \ \ +\|(3\phi^{2}-1)\nabla\Delta\phi\|_{L^{2}}^{2}
  +\|(3\phi^{2}-1)\nabla(\phi^{3}-\phi)\|_{L^{2}}^{2}
  \Big)\nonumber\\
  &\leq
  \frac{k\varepsilon\gamma}{8}\|\nabla\Delta^{2}\phi\|_{L^{2}}^{2}+C\Big(\|\phi\|_{L^{\infty}}^{2}\|\nabla\phi\|_{L^{6}}^{2}\|\Delta\phi\|_{L^{3}}^{2}
  +\|\phi\|_{L^{\infty}}^{4}\|\phi^{2}-1\|_{L^{\infty}}^{2}\|\nabla\phi\|_{L^{2}}^{2}\nonumber\\
  &\ \ \ +
  \|3\phi^{2}-1\|_{L^{\infty}}^{2}\|\nabla\Delta\phi\|_{L^{2}}^{2}+\|3\phi^{2}-1\|_{L^{\infty}}^{4}\|\nabla\phi\|_{L^{2}}^{2}\Big)
  \nonumber\\
  &\leq
  \frac{k\varepsilon\gamma}{8}\|\nabla\Delta^{2}\phi\|_{L^{2}}^{2}+C\big(\|\nabla\Delta\phi\|_{L^{2}}^{2}+1\big);
\end{align}
\begin{align}\label{eq3.9}
  I_{23}&=M_{2}\varepsilon\gamma(B(\phi)-\beta)\int_{Q}\nabla\Delta\phi\cdot\nabla\Delta^{2}\phi
  dx-\frac{M_{2}\gamma(B(\phi)-\beta)}{\varepsilon}\int_{Q}\nabla(\phi^{3}-\phi)\cdot\nabla\Delta^{2}\phi
  dx\nonumber\\
  &\leq
  \frac{k\varepsilon\gamma}{8}\|\nabla\Delta^{2}\phi\|_{L^{2}}^{2}+C(B(\phi)-\beta)^{2}\Big( \|\nabla\Delta\phi\|_{L^{2}}^{2}
  +\|\nabla(\phi^{3}-\phi)\|_{L^{2}}^{2}\Big)\nonumber\\
  &\leq
  \frac{k\varepsilon\gamma}{8}\|\nabla\Delta^{2}\phi\|_{L^{2}}^{2}+C\Big(\|\nabla\phi\|_{L^{2}}^{4}
  +\|\phi^{2}-1\|_{L^{2}}^{4}+1\Big)\Big(\|\nabla\Delta\phi\|_{L^{2}}^{2}+\|\phi^{2}-1\|_{L^{\infty}}^{2}\|\nabla\phi\|_{L^{2}}^{2}\Big)
  \nonumber\\
  &\leq
  \frac{k\varepsilon\gamma}{8}\|\nabla\Delta^{2}\phi\|_{L^{2}}^{2}+C\big(\|\nabla\Delta\phi\|_{L^{2}}^{2}+1\big).
\end{align}
From \eqref{eq3.7}-\eqref{eq3.9}, we get
\begin{equation}\label{eq3.10}
  I_{2}\leq
  -\frac{5k\varepsilon\gamma}{8}\|\nabla\Delta^{2}\phi\|_{L^{2}}^{2}+C\big(\|\nabla\Delta\phi\|_{L^{2}}^{2}+1\big).
\end{equation}
Combining the above estimates \eqref{eq3.5} and \eqref{eq3.10}, we
obtain
\begin{equation}\label{eq3.11}
  \frac{d}{dt}\|\nabla\Delta\phi\|_{L^{2}}^{2}+k\varepsilon\gamma\|\nabla\Delta^{2}\phi\|_{L^{2}}^{2}\leq
  C\big(\|\nabla u\|_{L^{2}}^{2}+1\big)\big(\|\nabla\Delta\phi\|_{L^{2}}^{2}+1\big).
\end{equation}
The Gronwall's equality yields that
\begin{equation}\label{eq3.12}
  \|\nabla\Delta\phi(t)\|_{L^{2}}^{2}\leq
  \|\nabla\Delta\phi_{0}\|_{L^{2}}^{2}\exp\Big(C\int_{0}^{t}(\|\nabla u(\tau)\|_{L^{2}}^{2}+1)d\tau\Big).
\end{equation}
The estimate \eqref{eq3.12} with \eqref{eq3.2} imply \eqref{eq3.3}
immediately. We complete the proof of Lemma \ref{le3.1}.
\end{proof}
\medskip

By \eqref{eq3.1} and \eqref{eq3.3},  for any $0<T<\infty$, we obtain
\begin{equation}\label{eq3.13}
  \sup_{0\leq t\leq T}\|\phi(\cdot,t)\|_{H^{3}}\leq C.
\end{equation}
By the Sobolev embedding $H^{2}_{per}(Q)\hookrightarrow
L^{\infty}_{per}(Q)$, \eqref{eq3.13} yields that
\begin{equation}\label{eq3.14}
  \sup_{0\leq t\leq T}\|\nabla\phi\|_{L^{\infty}}\leq C.
\end{equation}
This result will be used  frequently in the proofs of Theorems
\ref{th1.1} and \ref{th1.2}.

\section{The proof of Theorem \ref{th1.1}}
We argue Theorem \ref{th1.1} by contradiction. Assume that the
result \eqref{eq1.20} is not true, which means that there exists a
constant $M>0$ such that
\begin{equation}\label{eq4.1}
\int_{0}^{T_{*}}\frac{\|\omega(\cdot,t)\|_{\dot{B}^{0}_{\infty,\infty}}}{\sqrt{1+\ln(e+\|\omega(\cdot,t)\|_{\dot{B}^{0}_{\infty,\infty}})}}dt\leq
M.
\end{equation} Under the condition \eqref{eq4.1}, if we can prove that
\begin{equation}\label{eq4.2}
  \limsup_{t\nearrow T_{*}}\big(\|u(\cdot,t)\|_{H^{3}}+\|\phi(\cdot,t)\|_{H^{6}}\big)\leq C
\end{equation}
holds for some constant $C$ depending only on $u_{0}$, $\phi_{0}$,
$M$, $T_{*}$ and coefficients of the system
\eqref{eq1.1}--\eqref{eq1.5}, then we can extend the solution
$(u,\phi)$ beyond the time $t=T_{*}$, which leads to the
contradiction. Therefore, it suffices to show that under the
condition \eqref{eq4.1}, we get \eqref{eq4.2}.

\medskip

Taking the curl on \eqref{eq1.1}, we obtain
\begin{equation}\label{eq4.3}
  \partial_{t}\omega-\mu\Delta
  \omega+u\cdot\nabla\omega=\omega\cdot\nabla
  u+\nabla\times(\frac{\delta E}{\delta\phi}\nabla\phi).
\end{equation}
Multiplying \eqref{eq4.3} by $\omega$ and integrating over $Q$, we
have
\begin{equation}\label{eq4.4}
  \frac{1}{2}\frac{d}{dt}\|\omega\|_{L^{2}}^{2}+\mu\|\nabla\omega\|_{L^{2}}^{2}=\int_{Q}w\cdot\nabla
  u\cdot\omega dx-\int_{Q}\frac{\delta
  E}{\delta\phi}\nabla\phi\cdot\nabla\times\omega dx,
\end{equation}
where we have used the fact $\int_{Q}u\cdot\nabla\omega\cdot\omega
dx=0$ due to $\nabla\cdot u=0$. Since the Riesz operators are
bounded in $L^{2}$ and $\nabla
u=(-\Delta)^{-1}\nabla(\nabla\times\omega)$, we have $\|\nabla
u\|_{L^{2}}\leq C\|\omega\|_{L^{2}}$. This implies that
\begin{equation}\label{eq4.5}
  \Big{|}\int_{Q}w\cdot\nabla u\cdot\omega dx\Big{|}\leq
  C\|\omega\|_{L^{\infty}}\|\nabla u\|_{L^{2}}\|\omega\|_{L^{2}}\leq
  C\|\omega\|_{L^{\infty}}\|\omega\|_{L^{2}}^{2}.
\end{equation}
Applying Young's inequality and \eqref{eq3.14}, we have
\begin{align}\label{eq4.6}
  \Big{|}\int_{Q}\frac{\delta
  E}{\delta\phi}\nabla\phi\cdot\nabla\times\omega dx\Big{|}&\leq
  \frac{\mu}{4}\|\nabla\omega\|_{L^{2}}^{2}+C\|\frac{\delta
  E}{\delta\phi}\nabla\phi\|_{L^{2}}^{2}\nonumber\\
  &\leq
  \frac{\mu}{4}\|\nabla\omega\|_{L^{2}}^{2}+C\|\frac{\delta
  E}{\delta\phi}\|_{L^{2}}^{2}\|\nabla\phi\|_{L^{\infty}}^{2}\nonumber\\
  &\leq
  \frac{\mu}{4}\|\nabla\omega\|_{L^{2}}^{2}+C\|\frac{\delta
  E}{\delta\phi}\|_{L^{2}}^{2}.
\end{align}
Taking \eqref{eq4.5} and \eqref{eq4.6} into \eqref{eq4.4}, we obtain
\begin{equation}\label{eq4.7}
  \frac{d}{dt}\|\omega\|_{L^{2}}^{2}+\frac{3\mu}{2}\|\nabla\omega\|_{L^{2}}^{2}\leq
  C\Big(\|\omega\|_{L^{\infty}}+1\Big)\Big(\|\omega\|_{L^{2}}^{2}+\|\frac{\delta
  E}{\delta\phi}\|_{L^{2}}^{2}\Big).
\end{equation}
\medskip

On the other hand, after integration by parts, we obtain from
\eqref{eq1.11} that
\begin{align}\label{eq4.8}
  \frac{1}{2}\frac{d}{dt}&\|\frac{\delta
  E}{\delta\phi}\|_{L^{2}}^{2}=\int_{Q}\frac{\partial}{\partial t}\frac{\delta
  E}{\delta\phi}\cdot\frac{\delta
  E}{\delta\phi}dx\nonumber\\
  &=\int_{Q}\frac{\partial}{\partial
  t}\big[kg(\phi)+M_{1}(A(\phi)-\alpha)+M_{2}(B(\phi)-\beta)f(\phi)\big]\cdot \frac{\delta
  E}{\delta\phi}dx\nonumber\\
  &=-k\int_{Q}\frac{\partial}{\partial
  t}\Delta f(\phi)\cdot\frac{\delta
  E}{\delta\phi}dx+\frac{k}{\varepsilon^{2}}\int_{Q}\frac{\partial}{\partial
  t}\big[(3\phi^{2}-1)f(\phi)\big]\cdot\frac{\delta
  E}{\delta\phi}dx+M_{1}\frac{d}{dt}A(\phi)\int_{Q}\frac{\delta
  E}{\delta\phi}dx\nonumber\\
  &\ \ \ +M_{2}\frac{d}{dt}B(\phi)\int_{Q}f(\phi)\cdot\frac{\delta
  E}{\delta\phi}dx+M_{2}(B(\phi)-\beta)\int_{Q}\frac{\partial}{\partial
  t}f(\phi)\cdot\frac{\delta
  E}{\delta\phi}dx\nonumber\\
  &:=J_{1}+J_{2}+J_{3}+J_{4}+J_{5}.
\end{align}
Noticing from \eqref{eq1.3} that
\begin{align*}
  \|\frac{\partial\phi}{\partial
  t}\|_{L^{2}}&\leq C\Big(\|u\cdot\nabla\phi\|_{L^{2}}+\|\frac{\delta
  E}{\delta\phi}\|_{L^{2}}\Big)\\&\leq C\Big(\|u\|_{L^{2}}\|\nabla\phi\|_{L^{\infty}}+\|\frac{\delta
  E}{\delta\phi}\|_{L^{2}}\Big)\\&\leq C\Big(\|\frac{\delta
  E}{\delta\phi}\|_{L^{2}}+1\Big),
\end{align*}
\begin{align*}
  \|\nabla\frac{\partial\phi}{\partial
  t}\|_{L^{2}}&\leq C\Big(\|\nabla u\cdot\nabla\phi\|_{L^{2}}+\|u\cdot\nabla^{2}\phi\|_{L^{2}}+\|\nabla\frac{\delta
  E}{\delta\phi}\|_{L^{2}}\Big)\\
  &\leq C\Big(\|\nabla u\|_{L^{2}}\|\nabla \phi\|_{L^{\infty}}+\|u\|_{L^{3}}\|\nabla^{2}\phi\|_{L^{6}}+\|\nabla\frac{\delta
  E}{\delta\phi}\|_{L^{2}}\Big)\\
  &\leq C\Big(\|\Delta\frac{\delta
  E}{\delta\phi}\|_{L^{2}}+\|\nabla u\|_{L^{2}}+\|\frac{\delta
  E}{\delta\phi}\|_{L^{2}}+1\Big),
\end{align*}
\begin{align*}
  \|\frac{\partial f(\phi)}{\partial
  t}\|_{L^{2}} &\leq C\Big(\|\Delta\frac{\partial\phi}{\partial t}\|_{L^{2}}+\|\frac{\partial}{\partial
  t}(\phi^{3}-\phi)\|_{L^{2}}\Big)\\
  &\leq C\Big(\|\Delta u\cdot\nabla\phi\|_{L^{2}}+\|\nabla u\cdot\nabla^{2}\phi\|_{L^{2}}+\|u\cdot\nabla\Delta\phi\|_{L^{2}}+\|\Delta\frac{\delta
  E}{\delta\phi}\|_{L^{2}}+\|\frac{\partial}{\partial
  t}(\phi^{3}-\phi)\|_{L^{2}}\Big)\\
  &\leq C\Big(\|\nabla\phi\|_{L^{\infty}}\|\Delta u\|_{L^{2}}
  +\|\nabla u\|_{L^{3}}\|\nabla^{2}\phi\|_{L^{6}}
  +\|u\|_{L^{\infty}}\|\nabla\Delta\phi\|_{L^{2}}+\|\Delta\frac{\delta
  E}{\delta\phi}\|_{L^{2}}+\|\frac{\partial\phi}{\partial
  t}\|_{L^{2}}\Big)\\
  &\leq C\Big(\|\Delta u\|_{L^{2}}
  +\|\Delta\frac{\delta
  E}{\delta\phi}\|_{L^{2}}+\|\frac{\delta
  E}{\delta\phi}\|_{L^{2}}+1\Big).
\end{align*}
Then we can estimate $J_{i}$ $(i=1, 2, 3, 4, 5)$  as follows: For
$J_{1}$, we can further split it into the following two terms:
\begin{align}\label{eq4.9}
  J_{1}&=k\varepsilon\int_{Q}\Delta\frac{\partial \phi}{\partial
  t}\cdot\Delta\frac{\delta
  E}{\delta\phi}dx-\frac{k}{\varepsilon}\int_{Q}\frac{\partial}{\partial
  t}\Delta(\phi^{3}-\phi)\cdot\frac{\delta
  E}{\delta\phi}dx :=J_{11}+J_{12}.
\end{align}
By using Leibniz's rule, \eqref{eq1.3} yields that
\begin{align}\label{eq4.10}
  J_{11}&=-k\varepsilon\gamma\|\Delta\frac{\delta
  E}{\delta\phi}\|_{L^{2}}^{2}-k\varepsilon\int_{Q}\Delta(u\cdot\nabla\phi)\cdot\Delta\frac{\delta
  E}{\delta\phi}dx\nonumber\\
  &\leq -\frac{9k\varepsilon\gamma}{10}\|\Delta\frac{\delta
  E}{\delta\phi}\|_{L^{2}}^{2}+C\|\Delta(u\cdot\nabla\phi)\|_{L^{2}}^{2}\nonumber\\
  &\leq -\frac{9k\varepsilon\gamma}{10}\|\Delta\frac{\delta
  E}{\delta\phi}\|_{L^{2}}^{2}+C\Big(\|\Delta u\cdot\nabla\phi\|_{L^{2}}^{2}+2\|\nabla u\cdot\nabla^{2}\phi\|_{L^{2}}^{2}
  +\|u\cdot\nabla\Delta\phi\|_{L^{2}}^{2}\Big)\nonumber\\
  &\leq -\frac{9k\varepsilon\gamma}{10}\|\Delta\frac{\delta
  E}{\delta\phi}\|_{L^{2}}^{2}+C\Big(\|\nabla\phi\|_{L^{\infty}}^{2}\|\Delta u\|_{L^{2}}^{2}
  +\|\nabla u\|_{L^{3}}^{2}\|\nabla^{2}\phi\|_{L^{6}}^{2}
  +\|u\|_{L^{\infty}}^{2}\|\nabla\Delta\phi\|_{L^{2}}^{2}\Big)\nonumber\\
  &\leq -\frac{9k\varepsilon\gamma}{10}\|\Delta\frac{\delta
  E}{\delta\phi}\|_{L^{2}}^{2}+C\|\Delta
  u\|_{L^{2}}^{2}+C(\|\nabla u\|_{L^{2}}^{2}+1),
\end{align}
\begin{align}\label{eq4.11}
  J_{12}&=-\frac{k}{\varepsilon}\int_{Q}\frac{\partial}{\partial
  t}(\phi^{3}-\phi)\cdot \Delta\frac{\delta
  E}{\delta\phi}dx\nonumber\\
  &\leq \frac{k\varepsilon\gamma}{10}\|\Delta\frac{\delta
  E}{\delta\phi}\|_{L^{2}}^{2}+\|\frac{\partial}{\partial
  t}(\phi^{3}-\phi)\|_{L^{2}}^{2}\nonumber\\
  &\leq \frac{k\varepsilon\gamma}{10}\|\Delta\frac{\delta
  E}{\delta\phi}\|_{L^{2}}^{2}+C\Big(\|\phi\|_{L^{\infty}}^{4}\|\frac{\partial\phi}{\partial
  t}\|_{L^{2}}^{2}+\|\frac{\partial\phi}{\partial
  t}\|_{L^{2}}^{2}\Big)\nonumber\\
  &\leq \frac{k\varepsilon\gamma}{10}\|\Delta\frac{\delta
  E}{\delta\phi}\|_{L^{2}}^{2}+C\Big(\|\frac{\delta
  E}{\delta\phi}\|_{L^{2}}^{2}+1\Big).
\end{align}
 Hence, we infer from
\eqref{eq4.10} and \eqref{eq4.11} that
\begin{align}\label{eq4.12}
  J_{1}\leq -\frac{4k\varepsilon\gamma}{5}\|\Delta\frac{\delta
  E}{\delta\phi}\|_{L^{2}}^{2}+C\|\Delta
  u\|_{L^{2}}^{2}+C\Big(\|\nabla u\|_{L^{2}}^{2}+\|\frac{\delta
  E}{\delta\phi}\|_{L^{2}}^{2}+1\Big).
\end{align}
Similarly, we can estimate $J_{2}$, $J_{3}$, $J_{4}$ and $J_{5}$ as
follows:
\begin{align}\label{eq4.13}
  J_{2}&=\frac{6k}{\varepsilon^{2}}\int_{Q}\phi f(\phi)\frac{\partial\phi}{\partial
  t}\cdot\frac{\delta
  E}{\delta\phi}dx+\frac{k}{\varepsilon^{2}}\int_{Q}(3\phi^{2}-1)\frac{\partial f(\phi)}{\partial
  t}\cdot\frac{\delta
  E}{\delta\phi}dx\nonumber\\
  &=-\frac{6k}{\varepsilon}\int_{Q}\phi \Delta\phi\frac{\partial\phi}{\partial
  t}\cdot\frac{\delta
  E}{\delta\phi}dx+\frac{6k}{\varepsilon^{3}}\int_{Q}\phi^{2} (\phi^{2}-1)\frac{\partial\phi}{\partial
  t}\cdot\frac{\delta
  E}{\delta\phi}dx+\frac{k}{\varepsilon^{2}}\int_{Q}(3\phi^{2}-1)\frac{\partial
  f(\phi)}{\partial
  t}\cdot\frac{\delta
  E}{\delta\phi}dx\nonumber\\
  &\leq C\Big(\|\phi\|_{L^{\infty}}\|\Delta\phi\|_{L^{6}}\|\frac{\partial\phi}{\partial
  t}\|_{L^{3}}
  +\|\phi\|_{L^{\infty}}^{2}\|\phi^{2}-1\|_{L^{\infty}}\|\frac{\partial\phi}{\partial t}\|_{L^{2}}
  +\|\phi^{2}-1\|_{L^{\infty}}\|\frac{\partial f(\phi)}{\partial
  t}\|_{L^{2}}\Big)\|\frac{\delta E}{\delta \phi}\|_{L^{2}}\nonumber\\
  &\leq C\Big(\|\frac{\partial\phi}{\partial
  t}\|_{L^{2}}+\|\nabla\frac{\partial\phi}{\partial
  t}\|_{L^{2}}+\|\frac{\partial f(\phi)}{\partial
  t}\|_{L^{2}}\Big)\|\frac{\delta E}{\delta \phi}\|_{L^{2}}\nonumber\\
  &\leq C\Big(\|\Delta u\|_{L^{2}}
  +\|\Delta\frac{\delta
  E}{\delta\phi}\|_{L^{2}}+\|\nabla u\|_{L^{2}}+\|\frac{\delta
  E}{\delta\phi}\|_{L^{2}}+1\Big)\|\frac{\delta E}{\delta
  \phi}\|_{L^{2}}\nonumber\\
  &\leq \frac{k\varepsilon\gamma}{10}\|\Delta\frac{\delta
  E}{\delta\phi}\|_{L^{2}}^{2}+ C\|\Delta u\|_{L^{2}}^{2}+C\Big(\|\nabla u\|_{L^{2}}^{2}+\|\frac{\delta
  E}{\delta\phi}\|_{L^{2}}^{2}+1\Big),
\end{align}
\begin{align}\label{eq4.14}
  J_{3}&\leq C\Big{|}\frac{d A(\phi)}{dt}\Big{|}\Big{|}\int_{Q}\frac{\delta E}{\delta \phi}dx\Big{|}
  \leq C\Big{|}\int_{Q}\frac{\partial \phi}{\partial t}dx\Big{|}\Big{|}\int_{Q}\frac{\delta E}{\delta \phi}dx\Big{|}\nonumber\\
  &\leq C\|\frac{\partial \phi}{\partial t}\|_{L^{2}}\|\frac{\delta E}{\delta \phi}\|_{L^{2}}\leq C\Big(\|\frac{\delta
  E}{\delta\phi}\|_{L^{2}}^{2}+1\Big),
\end{align}
\begin{align}\label{eq4.15}
  J_{4}&\leq C\Big{|}\frac{d B(\phi)}{dt}\Big{|}\|f(\phi)\|_{L^{2}}\|\frac{\delta
  E}{\delta\phi}\|_{L^{2}}\nonumber\\
  &\leq C\Big(\|\nabla\phi\|_{L^{2}}\|\nabla\frac{\partial \phi}{\partial t}\|_{L^{2}}
  +\|\phi^{3}-\phi\|_{L^{\infty}}\|\frac{\partial \phi}{\partial
  t}\|_{L^{2}}\Big)\|f(\phi)\|_{L^{2}}\|\frac{\delta
  E}{\delta\phi}\|_{L^{2}}\nonumber\\
  &\leq C\Big(\|\Delta\frac{\delta
  E}{\delta\phi}\|_{L^{2}}+ \|\nabla u\|_{L^{2}}+\|\frac{\delta
  E}{\delta\phi}\|_{L^{2}}+1\Big)\|\frac{\delta
  E}{\delta\phi}\|_{L^{2}}\nonumber\\
  &\leq \frac{k\varepsilon\gamma}{10}\|\Delta\frac{\delta
  E}{\delta\phi}\|_{L^{2}}^{2}+C\Big(\|\nabla u\|_{L^{2}}^{2}+\|\frac{\delta
  E}{\delta\phi}\|_{L^{2}}^{2}+1\Big),
\end{align}
\begin{align}\label{eq4.16}
  J_{5}&\leq C|B(\phi)-\beta|\|\frac{\partial f(\phi)}{\partial t}\|_{L^{2}}\|\frac{\delta
  E}{\delta\phi}\|_{L^{2}}\nonumber\\
  &\leq C\Big(\|\nabla\phi\|_{L^{2}}^{2}+\|\phi^{2}-1\|_{L^{2}}^{2}\Big)\|\frac{\partial f(\phi)}{\partial t}\|_{L^{2}}\|\frac{\delta
  E}{\delta\phi}\|_{L^{2}}\nonumber\\
  &\leq C\Big(\|\Delta u\|_{L^{2}}
  +\|\Delta\frac{\delta
  E}{\delta\phi}\|_{L^{2}}+\|\nabla u\|_{L^{2}}+\|\frac{\delta
  E}{\delta\phi}\|_{L^{2}}+1\Big)\|\frac{\delta E}{\delta
  \phi}\|_{L^{2}}\nonumber\\
  &\leq \frac{k\varepsilon\gamma}{10}\|\Delta\frac{\delta
  E}{\delta\phi}\|_{L^{2}}^{2}+C\|\Delta u\|_{L^{2}}^{2}+C\Big(\|\nabla u\|_{L^{2}}^{2}+\|\frac{\delta
  E}{\delta\phi}\|_{L^{2}}^{2}+1\Big).
\end{align}
Taking \eqref{eq4.12}--\eqref{eq4.16} into \eqref{eq4.8}, by using
the fact that $\|\Delta u\|_{L^{2}}\leq C\|\nabla w\|_{L^{2}}$, we
conclude that
\begin{align}\label{eq4.17}
  \frac{d}{dt}\|\frac{\delta
  E}{\delta\phi}\|_{L^{2}}^{2}+k\varepsilon\gamma\|\Delta\frac{\delta
  E}{\delta\phi}\|_{L^{2}}^{2}\leq \tilde{C}\|\nabla w\|_{L^{2}}^{2}+C\Big(\|w\|_{L^{2}}^{2}+\|\frac{\delta
  E}{\delta\phi}\|_{L^{2}}^{2}+1\Big),
\end{align}
where $\tilde{C}$ is a constant depending only on
$\|u_{0}\|_{H^{1}}$, $\|\phi_{0}\|_{H^{3}}$, $T_{*}$ and
coefficients of the system due to the estimate \eqref{eq2.3}.

\medskip

Set
$$
  \eta=\frac{\mu}{2\tilde{C}}.
$$
Then multiplying \eqref{eq4.17} by $\eta$, adding \eqref{eq4.7}
together, applying Lemma \ref{le2.2} with $s=3$, we obtain
\begin{align}\label{eq4.18}
  \frac{d}{dt}\Big(\|w\|_{L^{2}}^{2}&+\eta\|\frac{\delta
  E}{\delta\phi}\|_{L^{2}}^{2}\Big)+\mu\|\nabla w\|_{L^{2}}^{2}+k\varepsilon\gamma\eta\|\Delta\frac{\delta
  E}{\delta\phi}\|_{L^{2}}^{2}\nonumber\\
  &\leq
  C\Big(\|w\|_{L^{\infty}}+1\Big)\Big(\|w\|_{L^{2}}^{2}+\eta\|\frac{\delta
  E}{\delta\phi}\|_{L^{2}}^{2}+1\Big)\nonumber\\
  &\leq
  C\Big(1+\|w\|_{\dot{B}^{0}_{\infty,\infty}}\sqrt{1+\ln(e+\|w\|_{\dot{B}^{0}_{\infty,\infty}})}\Big)\Big(\|w\|_{L^{2}}^{2}+\eta\|\frac{\delta
  E}{\delta\phi}\|_{L^{2}}^{2}+1\Big)\nonumber\\
  &\leq
  C\Big(\|w\|_{L^{2}}^{2}+\eta\|\frac{\delta
  E}{\delta\phi}\|_{L^{2}}^{2}+1\Big)+C\frac{\|w\|_{\dot{B}^{0}_{\infty,\infty}}}{\sqrt{1+\ln(e+\|w\|_{\dot{B}^{0}_{\infty,\infty}})}}
  \sqrt{1+\ln(e+\|w\|_{\dot{B}^{0}_{\infty,\infty}})}\nonumber\\
  &\ \ \ \times\sqrt{\ln(e+\|w\|_{H^{2}})}\Big(\|w\|_{L^{2}}^{2}+\eta\|\frac{\delta
  E}{\delta\phi}\|_{L^{2}}^{2}+1\Big)\nonumber\\
  &\leq C\Big(\|w\|_{L^{2}}^{2}+\eta\|\frac{\delta
  E}{\delta\phi}\|_{L^{2}}^{2}+1\Big)\nonumber\\
  &\ \ \ +C\frac{\|w\|_{\dot{B}^{0}_{\infty,\infty}}}{\sqrt{1+\ln(e+\|w\|_{\dot{B}^{0}_{\infty,\infty}})}}
  \ln(e+\|\nabla^{3}u\|_{L^{2}})\Big(\|w\|_{L^{2}}^{2}+\eta\|\frac{\delta
  E}{\delta\phi}\|_{L^{2}}^{2}+1\Big),
\end{align}
where $C$ is a constant which may depend on $\eta$.

\medskip

By the condition \eqref{eq1.20}, one concludes that for any small
constant $\sigma>0$, there exists $T_{0}<T$ such that
\begin{equation}\label{eq4.19}
  \int_{T_{0}}^{T}\frac{\|\omega\|_{\dot{B}^{0}_{\infty,\infty}}}{\sqrt{1+\ln(e+\|\omega\|_{\dot{B}^{0}_{\infty,\infty}})}}dt<\sigma.
\end{equation}
For any $T_{0}<t\leq T$, we set
\begin{equation}\label{eq4.20}
  h(t):=\sup_{T_{0}\leq \tau\leq t}\Big(\|\nabla\Delta
  u(\tau)\|_{L^{2}}^{2}+\hat{\eta}\|\Delta\frac{\delta E}{\delta
  \phi}(\tau)\|_{L^{2}}^{2}\Big),
\end{equation}
where $\hat{\eta}$ is a determined constant which specified later.
Applying Gronwall's inequality to \eqref{eq4.18} in the time
interval $[T_{0},t]$, one has
\begin{align}\label{eq4.21}
  \|w(t)\|_{L^{2}}^{2}&+\eta\|\frac{\delta
  E}{\delta\phi}(t)\|_{L^{2}}^{2}\nonumber\\
  &\leq C_{0}\exp\Big(\int_{T_{0}}^{t}Cds+C\ln(e+h(t))
  \int_{T_{0}}^{t}\frac{\|\omega\|_{\dot{B}^{0}_{\infty,\infty}}}{\sqrt{1+\ln(e+\|\omega\|_{\dot{B}^{0}_{\infty,\infty}})}}d\tau\Big)\nonumber\\
  &\leq C_{0}\exp\big(C(t-T_{0})+C\sigma\ln(e+h(t))\big)\nonumber\\
  &\leq C_{0}(e+h(t))^{2C\sigma},
\end{align}
where $C_{0}=\|w(T_{0})\|_{L^{2}}^{2}+\eta\|\frac{\delta
E}{\delta\phi}(T_{0})\|_{L^{2}}^{2}$ is a positive constant
depending on $T_{0}$.

\medskip

Now we are in a position to derive higher order energy estimates of
the solution. Taking $\nabla\Delta$ on \eqref{eq1.1}, multiplying
$\nabla\Delta u$ and integrating over $Q$, we obtain
\begin{align}\label{eq4.22}
  \frac{1}{2}\frac{d}{dt}\|\nabla\Delta
  u\|_{L^{2}}^{2}+\mu\|\Delta^{2}u\|_{L^{2}}^{2}&=-\int_{Q}\nabla\Delta(u\cdot\nabla
  u)\cdot\nabla\Delta u dx+\int_{Q}\nabla\Delta(\frac{\delta E}{\delta
  \phi}\nabla\phi)\cdot\nabla\Delta u dx\nonumber\\
  &:=\tilde{I}_{1}+\tilde{I}_{2}.
\end{align}
Since $\nabla\cdot u=0$, $\tilde{I}_{1}$ can be rewritten as
\begin{equation*}
  \tilde{I}_{1}=-\int_{Q}\big[\nabla\Delta(u\cdot\nabla
  u)-u\cdot\nabla\nabla\Delta u\big]\cdot\nabla\Delta u dx.
\end{equation*}
By using Lemma \ref{le2.1}, we can estimate $\tilde{I}_{1}$ as
follows:
\begin{align}\label{eq4.23}
  \tilde{I}_{1}&\leq C\|\nabla\Delta(u\cdot\nabla
  u)-u\cdot\nabla\nabla\Delta u\|_{L^{4/3}}\|\nabla\Delta
  u\|_{L^{4}}\leq C\|\nabla u\|_{L^{2}}\|\nabla\Delta
  u\|_{L^{4}}^{2}\nonumber\\
  &\leq C\|\nabla u\|_{L^{2}}^{7/6}\|\Delta^{2}u\|_{L^{2}}^{11/6}\leq \frac{\mu}{8}\|\Delta^{2}u\|_{L^{2}}^{2}+C\|\nabla u\|_{L^{2}}^{14}\nonumber\\
  &\leq \frac{\mu}{8}\|\Delta^{2}u\|_{L^{2}}^{2}+C\|w\|_{L^{2}}^{14},
\end{align}
where we have used the Gagliardo-Nirenberg inequality:
\begin{equation*}
  \|\nabla\Delta u\|_{L^{4}}\leq C \|\nabla u\|_{L^{2}}^{1/12} \|\Delta^{2} u\|_{L^{2}}^{11/12}.
\end{equation*}
For $\tilde{I}_{2}$, after integration by parts, by using
\eqref{eq3.3} and \eqref{eq3.14}, we obtain
\begin{align}\label{eq4.24}
  \tilde{I}_{2}&=-\int_{Q}\Delta(\frac{\delta E}{\delta
  \phi}\nabla\phi)\Delta^{2}u dx\nonumber\\
  &\leq \frac{\mu}{8}\|\Delta^{2}u\|_{L^{2}}^{2}+C\|\Delta(\frac{\delta E}{\delta
  \phi}\nabla\phi)\|_{L^{2}}^{2}\nonumber\\
  &\leq \frac{\mu}{8}\|\Delta^{2}u\|_{L^{2}}^{2}+C\Big(\|\Delta\frac{\delta E}{\delta
  \phi}\nabla\phi\|_{L^{2}}^{2}+2\|\nabla\frac{\delta E}{\delta
  \phi}\nabla^{2}\phi\|_{L^{2}}^{2}+\|\frac{\delta E}{\delta
  \phi}\nabla\Delta\phi\|_{L^{2}}^{2}\Big)\nonumber\\
  &\leq \frac{\mu}{8}\|\Delta^{2}u\|_{L^{2}}^{2}+C\Big(\|\Delta\frac{\delta E}{\delta
  \phi}\|_{L^{2}}^{2}\|\nabla\phi\|_{L^{\infty}}^{2}+\|\nabla\frac{\delta E}{\delta
  \phi}\|_{L^{3}}^{2}\|\nabla^{2}\phi\|_{L^{6}}^{2}+\|\frac{\delta E}{\delta
  \phi}\|_{L^{\infty}}^{2}\|\nabla\Delta\phi\|_{L^{2}}^{2}\Big)\nonumber\\
  &\leq \frac{\mu}{8}\|\Delta^{2}u\|_{L^{2}}^{2}+C\Big(\|\Delta\frac{\delta E}{\delta
  \phi}\|_{L^{2}}^{2}+\|\frac{\delta E}{\delta
  \phi}\|_{L^{2}}^{2}+1\Big).
\end{align}
Combining \eqref{eq4.22}--\eqref{eq4.24}, we deduce that
\begin{align}\label{eq4.25}
  \frac{d}{dt}\|\nabla\Delta
  u\|_{L^{2}}^{2}+\frac{3\mu}{2}\|\Delta^{2}u\|_{L^{2}}^{2}\leq C\Big(\|w\|_{L^{2}}^{14}+\|\Delta\frac{\delta E}{\delta
  \phi}\|_{L^{2}}^{2}+\|\frac{\delta E}{\delta
  \phi}\|_{L^{2}}^{2}+1\Big).
\end{align}

To obtain the desired estimates for $\phi$, we start from
\eqref{eq1.11} that
\begin{align}\label{eq4.26}
 \frac{1}{2}\frac{d}{dt}\|\Delta \frac{\delta
  E}{\delta\phi}\|_{L^{2}}^{2}&=\int_{Q}\frac{\partial}{\partial t}\Delta \frac{\delta
  E}{\delta\phi}\cdot\Delta \frac{\delta
  E}{\delta\phi}dx\nonumber\\
  &=\int_{Q}\frac{\partial}{\partial
  t}\Delta\big[kg(\phi)+M_{1}(A(\phi)-\alpha)+M_{2}(B(\phi)-\beta)f(\phi)\big]\cdot\Delta \frac{\delta
  E}{\delta\phi}dx\nonumber\\
  &=\int_{Q}\frac{\partial}{\partial
  t}\Delta\big[kg(\phi)+M_{2}(B(\phi)-\beta)f(\phi)\big]\cdot\Delta \frac{\delta
  E}{\delta\phi}dx\nonumber\\
  &=\int_{Q}\frac{\partial}{\partial
  t}\big[kg(\phi)+M_{2}(B(\phi)-\beta)f(\phi)\big]\cdot\Delta^{2} \frac{\delta
  E}{\delta\phi}dx\nonumber\\
  &=-k\int_{Q}\frac{\partial}{\partial
  t}\Delta f(\phi)\cdot\Delta^{2} \frac{\delta
  E}{\delta\phi}dx+\frac{k}{\varepsilon^{2}}\int_{Q}\frac{\partial}{\partial
  t}\big[(3\phi^{2}-1)f(\phi)\big]\cdot\Delta^{2} \frac{\delta
  E}{\delta\phi}dx\nonumber\\
  &\ \ \ +M_{2}\frac{d}{dt}B(\phi)\int_{Q}f(\phi)\cdot\Delta^{2}\frac{\delta
  E}{\delta\phi}dx+M_{2}(B(\phi)-\beta)\int_{Q}\frac{\partial}{\partial
  t}f(\phi)\cdot\Delta^{2}\frac{\delta
  E}{\delta\phi}dx\nonumber\\
  &:=\tilde{J}_{1}+\tilde{J}_{2}+\tilde{J}_{3}+\tilde{J}_{4}.
\end{align}
Let us estimate the terms $\tilde{J}_{i}$ $(i=1,2,3,4)$ one by one.
For $\tilde{J}_{1}$, we divide it  into the following two parts:
\begin{align}\label{eq4.27}
  \tilde{J}_{1}&=k\varepsilon\int_{Q}\Delta^{2}\frac{\partial \phi}{\partial
  t}\cdot\Delta^{2} \frac{\delta
  E}{\delta\phi}dx-\frac{k}{\varepsilon}\int_{Q}\frac{\partial}{\partial
  t}\Delta(\phi^{3}-\phi)\cdot\Delta^{2} \frac{\delta
  E}{\delta\phi}dx :=\tilde{J}_{11}+\tilde{J}_{12}.
\end{align}
For $\tilde{J}_{11}$, by using Leibniz's rule, we deduce from \eqref{eq1.3}
that
\begin{align}\label{eq4.28}
  \tilde{J}_{11}&=-k\varepsilon\gamma\|\Delta^{2}\frac{\delta
  E}{\delta\phi}\|_{L^{2}}^{2}-k\varepsilon\int_{Q}\Delta^{2}(u\cdot\nabla\phi)\cdot\Delta^{2}\frac{\delta
  E}{\delta\phi}dx\nonumber\\
  &\leq -\frac{15k\varepsilon\gamma}{16}\|\Delta^{2}\frac{\delta
  E}{\delta\phi}\|_{L^{2}}^{2}+C\|\Delta^{2}(u\cdot\nabla\phi)\|_{L^{2}}^{2}\nonumber\\
  &\leq -\frac{15k\varepsilon\gamma}{16}\|\Delta^{2}\frac{\delta
  E}{\delta\phi}\|_{L^{2}}^{2}+C\Big(\|\Delta^{2}u\cdot\nabla\phi\|_{L^{2}}^{2}+4\|\nabla\Delta u\cdot\nabla^{2}\phi\|_{L^{2}}^{2}\nonumber\\
  &\ \ \ +6\|\Delta u\cdot\nabla\Delta\phi\|_{L^{2}}^{2}+4\|\nabla u\cdot\nabla^{2}\Delta\phi\|_{L^{2}}^{2}
  +\|u\cdot\nabla\Delta^{2}\phi\|_{L^{2}}^{2}\Big)\nonumber\\
  &\leq -\frac{15k\varepsilon\gamma}{16}\|\Delta^{2}\frac{\delta
  E}{\delta\phi}\|_{L^{2}}^{2}+C\Big(\|\nabla\phi\|_{L^{\infty}}^{2}\|\Delta^{2}u\|_{L^{2}}^{2}+\|\nabla\Delta
  u\|_{L^{3}}^{2}\|\nabla^{2}\phi\|_{L^{6}}^{2}\nonumber\\
  &\ \ \ +\|\Delta
  u\|_{L^{6}}^{2}\|\nabla\Delta\phi\|_{L^{3}}^{2}+\|\nabla
  u\|_{L^{\infty}}^{2}\|\Delta^{2}\phi\|_{L^{2}}^{2}+\|u\|_{L^{3}}^{2}\|\nabla^{5}\phi\|_{L^{6}}^{2}\Big)\nonumber\\
  &\leq -\frac{15k\varepsilon\gamma}{16}\|\Delta^{2}\frac{\delta
  E}{\delta\phi}\|_{L^{2}}^{2}+C\|\Delta^{2}u\|_{L^{2}}^{2}+C\Big[(\|\frac{\delta E}{\delta\phi}\|_{L^{2}}^{2}+1)(\|\nabla\Delta
  u\|_{L^{2}}^{2}+1)\nonumber\\
  &\ \ \ +(\|\nabla u\|_{L^{2}}^{2}+1)(\|\Delta\frac{\delta
  E}{\delta\phi}\|_{L^{2}}^{2}+1)\Big]\nonumber\\
  &\leq -\frac{15k\varepsilon\gamma}{16}\|\Delta^{2}\frac{\delta
  E}{\delta\phi}\|_{L^{2}}^{2}+C\|\Delta^{2}u\|_{L^{2}}^{2}\nonumber\\
  &\ \ \ +C\Big(\|\nabla u\|_{L^{2}}^{2}+\|\frac{\delta E}{\delta\phi}\|_{L^{2}}^{2}+1\Big)
  \Big(\|\nabla\Delta
  u\|_{L^{2}}^{2}+\|\Delta\frac{\delta
  E}{\delta\phi}\|_{L^{2}}^{2}+1\Big),
\end{align}
where we have used the facts $\|\Delta^{2}\phi\|_{L^{2}}^{2}\leq
C(\|\frac{\delta E}{\delta\phi}\|_{L^{2}}^{2}+1)$ and
$\|\nabla^{5}\phi\|_{L^{6}}^{2}\leq
C\|\nabla^{6}\phi\|_{L^{2}}^{2}\leq C(\|\Delta\frac{\delta
E}{\delta\phi}\|_{L^{2}}^{2}+1)$.  For $\tilde{J}_{12}$, it clear
that
\begin{align}\label{eq4.29}
  \tilde{J}_{12}&=-\frac{k}{\varepsilon}\int_{Q}\frac{\partial}{\partial
  t}\Delta(\phi^{3}-\phi)\cdot\Delta^{2} \frac{\delta
  E}{\delta\phi}dx
  =-\frac{6k}{\varepsilon}\int_{Q}\frac{\partial\big(|\nabla\phi|^{2}\phi\big)}{\partial
  t}\cdot\Delta^{2}\frac{\delta
  E}{\delta\phi}dx\nonumber\\&\ \ \ -\frac{3k}{\varepsilon}\int_{Q}\frac{\partial\big(\phi^{2}\Delta\phi\big)}{\partial
  t}\cdot\Delta^{2}\frac{\delta
  E}{\delta\phi}dx+\frac{k}{\varepsilon}\int_{Q}\Delta\frac{\partial\phi}{\partial
  t}\cdot\Delta^{2}\frac{\delta
  E}{\delta\phi}dx\nonumber\\
  &=-\frac{12k}{\varepsilon}\int_{Q}\phi\nabla\phi\nabla\frac{\partial \phi}{\partial t}\cdot\Delta^{2}\frac{\delta
  E}{\delta\phi}dx-\frac{6k}{\varepsilon}\int_{Q}|\nabla\phi|^{2}\frac{\partial\phi}{\partial
  t}\cdot\Delta^{2}\frac{\delta
  E}{\delta\phi}dx\nonumber\\
  &\ \ \ -\frac{6k}{\varepsilon}\int_{Q}\phi\Delta\phi\frac{\partial\phi}{\partial
  t}\cdot\Delta^{2}\frac{\delta
  E}{\delta\phi}dx-\frac{3k}{\varepsilon}\int_{Q}\phi^{2}\Delta\frac{\partial\phi}{\partial
  t}\cdot\Delta^{2}\frac{\delta
  E}{\delta\phi}dx\nonumber\\
  &\ \ \ +\frac{k}{\varepsilon}\int_{Q}\Delta\frac{\partial\phi}{\partial
  t}\cdot\Delta^{2}\frac{\delta
  E}{\delta\phi}dx:=\sum_{i=1}^{5}\tilde{J}_{12i}.
\end{align}
Similarly, we can estimate the terms $\tilde{J}_{12i}$
$(i=1,2,3,4,5)$ as follows:
\begin{align}\label{eq4.30}
  \tilde{J}_{121}&\leq \frac{k\varepsilon\gamma}{16}\|\Delta^{2}\frac{\delta
  E}{\delta\phi}\|_{L^{2}}^{2}+C\|\phi\nabla\phi\nabla\frac{\partial\phi}{\partial
  t}\|_{L^{2}}^{2}\nonumber\\
  &\leq \frac{k\varepsilon\gamma}{16}\|\Delta^{2}\frac{\delta
  E}{\delta\phi}\|_{L^{2}}^{2}+C\|\phi\|_{L^{\infty}}^{2}\|\nabla\phi\|_{L^{\infty}}^{2}\|\nabla\frac{\partial\phi}{\partial
  t}\|_{L^{2}}^{2}\nonumber\\
  &\leq \frac{k\varepsilon\gamma}{16}\|\Delta^{2}\frac{\delta
  E}{\delta\phi}\|_{L^{2}}^{2}+C\Big(\|\Delta\frac{\delta E}{\delta\phi}\|_{L^{2}}^{2}+\|\nabla
  u\|_{L^{2}}^{2} +\|\frac{\delta
  E}{\delta\phi}\|_{L^{2}}^{2}+1\Big),
\end{align}
\begin{align}\label{eq4.31}
  \tilde{J}_{122}&\leq \frac{k\varepsilon\gamma}{16}\|\Delta^{2}\frac{\delta
  E}{\delta\phi}\|_{L^{2}}^{2}+C\|\nabla\phi\|_{L^{\infty}}^{4}\|\frac{\partial\phi}{\partial
  t}\|_{L^{2}}^{2}\nonumber\\
  &\leq \frac{k\varepsilon\gamma}{16}\|\Delta^{2}\frac{\delta
  E}{\delta\phi}\|_{L^{2}}^{2}+C\Big(\|\frac{\delta
  E}{\delta\phi}\|_{L^{2}}^{2}+1\Big),
\end{align}
\begin{align}\label{eq4.32}
  \tilde{J}_{123}&\leq \frac{k\varepsilon\gamma}{16}\|\Delta^{2}\frac{\delta
  E}{\delta\phi}\|_{L^{2}}^{2}+C\|\phi\|_{L^{\infty}}^{2}\|\Delta\phi\|_{L^{6}}^{2}\|\frac{\partial\phi}{\partial
  t}\|_{L^{3}}^{2}\nonumber\\
  &\leq \frac{k\varepsilon\gamma}{16}\|\Delta^{2}\frac{\delta
  E}{\delta\phi}\|_{L^{2}}^{2}+C\Big(\|u\cdot\nabla\phi\|_{L^{3}}^{2}+\|\frac{\delta
  E}{\delta\phi}\|_{L^{3}}^{2}\Big)\nonumber\\
  &\leq \frac{k\varepsilon\gamma}{16}\|\Delta^{2}\frac{\delta
  E}{\delta\phi}\|_{L^{2}}^{2}+C\Big(\|u\|_{L^{6}}^{2}\|\nabla\phi\|_{L^{6}}^{2}+\|\Delta\frac{\delta
  E}{\delta\phi}\|_{L^{2}}^{2}+\|\frac{\delta
  E}{\delta\phi}\|_{L^{2}}^{2}\Big)\nonumber\\
  &\leq \frac{k\varepsilon\gamma}{16}\|\Delta^{2}\frac{\delta
  E}{\delta\phi}\|_{L^{2}}^{2}+C\Big(\|\nabla u\|_{L^{2}}^{2}+\|\Delta\frac{\delta
  E}{\delta\phi}\|_{L^{2}}^{2}+\|\frac{\delta
  E}{\delta\phi}\|_{L^{2}}^{2}\Big),
\end{align}
\begin{align}\label{eq4.33}
  \tilde{J}_{124}&\leq \frac{k\varepsilon\gamma}{16}\|\Delta^{2}\frac{\delta
  E}{\delta\phi}\|_{L^{2}}^{2}+C\|\phi\|_{L^{\infty}}^{2}\|\Delta\frac{\partial\phi}{\partial
  t}\|_{L^{2}}^{2}\nonumber\\
  &\leq \frac{k\varepsilon\gamma}{16}\|\Delta^{2}\frac{\delta
  E}{\delta\phi}\|_{L^{2}}^{2}+C\Big(\|\Delta(u\cdot\nabla\phi)\|_{L^{2}}^{2}+\|\Delta\frac{\delta
  E}{\delta\phi}\|_{L^{2}}^{2}\Big)\nonumber\\
  &\leq \frac{k\varepsilon\gamma}{16}\|\Delta^{2}\frac{\delta
  E}{\delta\phi}\|_{L^{2}}^{2}+C\Big(\|\nabla\phi\|_{L^{\infty}}^{2}\|\Delta u\|_{L^{2}}^{2}+\|\nabla
  u\|_{L^{3}}^{2}\|\nabla^{2}\phi\|_{L^{6}}^{2}\nonumber\\
  &\ \ \ +\|u\|_{L^{\infty}}^{2}\|\nabla\Delta\phi\|_{L^{2}}^{2}+
  \|\Delta\frac{\delta
  E}{\delta\phi}\|_{L^{2}}^{2}\Big)\nonumber\\
  &\leq \frac{k\varepsilon\gamma}{16}\|\Delta^{2}\frac{\delta
  E}{\delta\phi}\|_{L^{2}}^{2}+C\Big(\|\Delta u\|_{L^{2}}^{2}+
  \|\Delta\frac{\delta
  E}{\delta\phi}\|_{L^{2}}^{2}+1\Big),
\end{align}
\begin{align}\label{eq4.34}
  \tilde{J}_{125}&\leq \frac{k\varepsilon\gamma}{16}\|\Delta^{2}\frac{\delta
  E}{\delta\phi}\|_{L^{2}}^{2}+C\|\Delta\frac{\partial\phi}{\partial
  t}\|_{L^{2}}^{2}\nonumber\\
  &\leq \frac{k\varepsilon\gamma}{16}\|\Delta^{2}\frac{\delta
  E}{\delta\phi}\|_{L^{2}}^{2}+C\Big(\|\Delta u\|_{L^{2}}^{2}+
  \|\Delta\frac{\delta
  E}{\delta\phi}\|_{L^{2}}^{2}+1\Big).
\end{align}
Putting estimates \eqref{eq4.30}--\eqref{eq4.34} together, we obtain
from \eqref{eq4.29} that
\begin{align}\label{eq4.35}
  \tilde{J}_{12}&\leq \frac{5k\varepsilon\gamma}{16}\|\Delta^{2}\frac{\delta
  E}{\delta\phi}\|_{L^{2}}^{2}+C\Big(\|\Delta u\|_{L^{2}}^{2}+
  \|\Delta\frac{\delta
  E}{\delta\phi}\|_{L^{2}}^{2}+\|\frac{\delta
  E}{\delta\phi}\|_{L^{2}}^{2}+1\Big).
\end{align}
Since $\|\Delta u\|_{L^{2}}^{2}\leq C(\|\nabla\Delta
u\|_{L^{2}}^{2}+1)$, we obtain from \eqref{eq4.35}, \eqref{eq4.27}
and \eqref{eq4.28} that
\begin{align}\label{eq4.36}
  \tilde{J}_{1}&\leq -\frac{5k\varepsilon\gamma}{8}\|\Delta^{2}\frac{\delta
  E}{\delta\phi}\|_{L^{2}}^{2}+C\|\Delta^{2}u\|_{L^{2}}^{2}\nonumber\\
  &\ \ \ +C\Big(\|\nabla u\|_{L^{2}}^{2}+\|\frac{\delta E}{\delta\phi}\|_{L^{2}}^{2}+1\Big)
  \Big(\|\nabla\Delta u\|_{L^{2}}^{2}+
  \|\Delta\frac{\delta
  E}{\delta\phi}\|_{L^{2}}^{2}+1\Big).
\end{align}
The estimates of $\tilde{J}_{i}$ $(i=2,3,4)$ can be proceeded as
that of $J_{i}$ $(i=2,3,4,5)$, thus we have
\begin{align}\label{eq4.37}
  \tilde{J}_{2}&=\frac{6k}{\varepsilon^{2}}\int_{Q}\phi f(\phi)\frac{\partial\phi}{\partial
  t}\cdot\Delta^{2} \frac{\delta
  E}{\delta\phi}dx+\frac{k}{\varepsilon^{2}}\int_{Q}(3\phi^{2}-1)\frac{\partial f(\phi)}{\partial
  t}\cdot\Delta^{2} \frac{\delta
  E}{\delta\phi}dx\nonumber\\
  &=-\frac{6k}{\varepsilon}\int_{Q}\phi \Delta\phi\frac{\partial\phi}{\partial
  t}\cdot\Delta^{2} \frac{\delta
  E}{\delta\phi}dx+\frac{6k}{\varepsilon^{3}}\int_{Q}\phi^{2} (\phi^{2}-1)\frac{\partial\phi}{\partial
  t}\cdot\Delta^{2} \frac{\delta
  E}{\delta\phi}dx\nonumber\\
  &\ \ \ +\frac{k}{\varepsilon^{2}}\int_{Q}(3\phi^{2}-1)\frac{\partial f(\phi)}{\partial
  t}\cdot\Delta^{2} \frac{\delta
  E}{\delta\phi}dx\nonumber\\
  &\leq \frac{k\varepsilon\gamma}{8}\|\Delta^{2}\frac{\delta
  E}{\delta\phi}\|_{L^{2}}^{2}+C\Big(\|\phi\|_{L^{\infty}}^{2}\|\Delta\phi\|_{L^{6}}^{2}\|\frac{\partial\phi}{\partial
  t}\|_{L^{3}}^{2}
  +\|\phi\|_{L^{\infty}}^{4}\|\phi^{2}-1\|_{L^{\infty}}^{2}\|\frac{\partial\phi}{\partial t}\|_{L^{2}}^{2}\nonumber\\
  &\ \ \ +\|3\phi^{2}-1\|_{L^{\infty}}^{2}\|\frac{\partial f(\phi)}{\partial
  t}\|_{L^{2}}^{2}\Big)\nonumber\\
  &\leq \frac{k\varepsilon\gamma}{8}\|\Delta^{2}\frac{\delta
  E}{\delta\phi}\|_{L^{2}}^{2}+C\Big(\|\frac{\partial\phi}{\partial
  t}\|_{L^{3}}^{2}+\|\Delta\frac{\partial \phi}{\partial
  t}\|_{L^{2}}^{2}+\|\frac{\partial\phi}{\partial t}\|_{L^{2}}^{2}\Big)\nonumber\\
  &\leq \frac{k\varepsilon\gamma}{8}\|\Delta^{2}\frac{\delta
  E}{\delta\phi}\|_{L^{2}}^{2}+C\Big(\|\nabla\Delta u\|_{L^{2}}^{2}+\|\Delta\frac{\delta E}{\delta \phi}\|_{L^{2}}^{2}
  +\|\frac{\delta E}{\delta \phi}\|_{L^{2}}^{2}+1\Big),
\end{align}
\begin{align}\label{eq4.38}
  \tilde{J}_{3}&\leq C|\frac{d B(\phi)}{dt}|\|f(\phi)\|_{L^{2}}\|\Delta^{2}\frac{\delta
  E}{\delta\phi}\|_{L^{2}}\nonumber\\
  &\leq \frac{k\varepsilon\gamma}{8}\|\Delta^{2}\frac{\delta
  E}{\delta\phi}\|_{L^{2}}^{2}+C\Big(\|\nabla\phi\|_{L^{2}}\|\nabla\frac{\partial \phi}{\partial t}\|_{L^{2}}
  +\|\phi^{3}-\phi\|_{L^{2}}\|\frac{\partial \phi}{\partial
  t}\|_{L^{2}}\Big)^{2}\|f(\phi)\|_{L^{2}}^{2}\nonumber\\
  &\leq \frac{k\varepsilon\gamma}{8}\|\Delta^{2}\frac{\delta
  E}{\delta\phi}\|_{L^{2}}^{2}+C\Big(\|\nabla u\|_{L^{2}}^{2}+\|\Delta\frac{\delta E}{\delta \phi}\|_{L^{2}}^{2}
  +\|\frac{\delta E}{\delta \phi}\|_{L^{2}}^{2}+1\Big),
\end{align}
\begin{align}\label{eq4.39}
  \tilde{J}_{4}&\leq C|B(\phi)-\beta|\|\frac{\partial f(\phi)}{\partial t}\|_{L^{2}}\|\Delta^{2}\frac{\delta
  E}{\delta\phi}\|_{L^{2}}\nonumber\\
  &\leq \frac{k\varepsilon\gamma}{8}\|\Delta^{2}\frac{\delta
  E}{\delta\phi}\|_{L^{2}}^{2}+C\Big(\|\nabla\phi\|_{L^{2}}^{2}+\|\phi^{2}-1\|_{L^{2}}^{2}\Big)^{2}\|\frac{\partial f(\phi)}{\partial t}\|_{L^{2}}^{2}\nonumber\\
  &\leq \frac{k\varepsilon\gamma}{8}\|\Delta^{2}\frac{\delta
  E}{\delta\phi}\|_{L^{2}}^{2}+C\Big(\|\Delta\frac{\partial \phi}{\partial t}\|_{L^{2}}^{2}+\|\frac{\partial \phi}{\partial t}\|_{L^{2}}^{2}\Big)\nonumber\\
  &\leq \frac{k\varepsilon\gamma}{8}\|\Delta^{2}\frac{\delta
  E}{\delta\phi}\|_{L^{2}}^{2}+C\Big(\|\nabla\Delta u\|_{L^{2}}^{2}+\|\Delta\frac{\delta E}{\delta \phi}\|_{L^{2}}^{2}
  +\|\frac{\delta E}{\delta \phi}\|_{L^{2}}^{2}+1\Big).
\end{align}
Taking \eqref{eq4.36}--\eqref{eq4.39} into \eqref{eq4.26}, we
conclude that
\begin{align}\label{eq4.40}
  \frac{d}{dt}\|\Delta \frac{\delta
  E}{\delta\phi}\|_{L^{2}}^{2}&+\frac{k\varepsilon\gamma}{2}\|\Delta^{2}\frac{\delta
  E}{\delta\phi}\|_{L^{2}}^{2}\leq \hat{C}\|\Delta^{2}u\|_{L^{2}}^{2}\nonumber\\
  &\ \ \ +C\Big(\|\nabla u\|_{L^{2}}^{2}+\|\frac{\delta E}{\delta\phi}\|_{L^{2}}^{2}+1\Big)
  \Big(\|\nabla\Delta u\|_{L^{2}}^{2}+\|\Delta\frac{\delta
  E}{\delta\phi}\|_{L^{2}}^{2}+1\Big),
\end{align}
where $\hat{C}$ is a constant depending only on $\|u_{0}\|_{H^{1}}$,
$\|\phi_{0}\|_{H^{3}}$, $T_{*}$ and coefficients of the system due
to the estimate \eqref{eq3.3}.

Set
$$
  \hat{\eta}=\frac{\mu}{2\hat{C}}.
$$
Multiplying \eqref{eq4.40} by $\hat{\eta}$ and adding the resultant
to \eqref{eq4.25}, we obtain
\begin{align}\label{eq4.41}
  \frac{d}{dt}\Big(\|\nabla\Delta
  u\|_{L^{2}}^{2}&+\hat{\eta}\|\Delta \frac{\delta
  E}{\delta\phi}\|_{L^{2}}^{2}\Big)+\mu\|\Delta^{2}u\|_{L^{2}}^{2}+\frac{k\varepsilon\gamma\hat{\eta}}{2}\|\Delta^{2}\frac{\delta
  E}{\delta\phi}\|_{L^{2}}^{2}\nonumber\\
  &\leq C\Big(\|\nabla u\|_{L^{2}}^{2}+\|\frac{\delta E}{\delta\phi}\|_{L^{2}}^{2}+1\Big)
  \Big(\|w\|_{L^{2}}^{14}+\|\nabla \Delta u\|_{L^{2}}^{2}+\hat{\eta}\|\Delta\frac{\delta
  E}{\delta\phi}\|_{L^{2}}^{2}+e\Big).
\end{align}
It follows from \eqref{eq4.21} that
\begin{align}\label{eq4.42}
  \frac{d}{dt}(e+h(t))\leq C\Big(\|\nabla u\|_{L^{2}}^{2}+\|\frac{\delta E}{\delta\phi}\|_{L^{2}}^{2}+1\Big)
  \Big((e+h(t))^{14C\sigma}+h(t)+e\Big).
\end{align}
Choosing $\sigma$ small enough such that $14C\sigma\leq 1$, we get
\begin{align}\label{eq4.43}
  \frac{d}{dt}(e+h(t))\leq C\Big(\|\nabla u\|_{L^{2}}^{2}+\|\frac{\delta E}{\delta\phi}\|_{L^{2}}^{2}+1\Big)
  (e+h(t)).
\end{align}
Applying Gronwall's inequality leads to
\begin{align}\label{eq4.44}
  h(t)\leq (e+h(T_{0}))\exp\Big(C\int_{T_{0}}^{t}\big(\|\nabla u\|_{L^{2}}^{2}+\|\frac{\delta E}{\delta\phi}\|_{L^{2}}^{2}+1\big)d\tau
  \Big).
\end{align}
This combines with the basic energy \eqref{eq1.12} yield the
boundness of $h(t)$ on the time interval $[T_{0},T]$. Since it is
easy to verify that
\begin{equation*}
  \|\phi\|_{H^{6}}^{2}\leq C(\|\Delta\frac{\delta
  E}{\delta\phi}\|_{L^{2}}^{2}+1),
\end{equation*}
we finally obtain from \eqref{eq4.41} that
$$
  \sup_{T_{0}\leq t\leq T}\big(\|u(\cdot,t)\|_{H^{3}}+\|\phi(\cdot,t)\|_{H^{6}}\big)\leq
  C.
$$
This completes the proof of Theorem \ref{th1.1}.

\section{The proof of Theorem \ref{th1.2}}
Similarly we prove Theorem \ref{th1.2} by contradiction. It suffices
to prove that if
\begin{equation}\label{eq5.1}
\int_{0}^{T_{*}}\frac{\|\omega(\cdot,t)\|_{\dot{B}^{-1}_{\infty,\infty}}^{2}}{1+\ln(e+\|\omega(\cdot,t)\|_{\dot{B}^{-1}_{\infty,\infty}})}dt\leq
M<\infty,
\end{equation}
then
\begin{equation}\label{eq5.2}
  \limsup_{t\nearrow T_{*}}\big(\|u(\cdot,t)\|_{H^{3}}+\|\phi(\cdot,t)\|_{H^{6}}\big)\leq C
\end{equation}
for some constant depending only on $u_{0}$, $\phi_{0}$, $M$,
$T_{*}$ and coefficients of the system \eqref{eq1.1}--\eqref{eq1.5}.

\medskip

Multiplying \eqref{eq4.3} by $\omega$ and integrating over $Q$, we
have
\begin{equation}\label{eq5.3}
  \frac{1}{2}\frac{d}{dt}\|\omega\|_{L^{2}}^{2}+\mu\|\nabla\omega\|_{L^{2}}^{2}=\int_{Q}w\cdot\nabla
  u\cdot\omega dx-\int_{Q}\frac{\delta
  E}{\delta\phi}\nabla\phi\cdot\nabla\times\omega dx,
\end{equation}
Since $\|\nabla u\|_{L^{2}}\leq C\|\omega\|_{L^{2}}$, by using Lemma
\ref{le2.3}, we obtain
\begin{align}\label{eq5.4}
  \Big{|}\int_{Q}w\cdot\nabla u\cdot\omega dx\Big{|}&\leq
  C\|\omega\|_{L^{4}}^{2}\|\nabla u\|_{L^{2}}\leq
  C\|\omega\|_{L^{4}}^{2}\|\omega\|_{L^{2}}\nonumber\\
  &\leq C\|w\|_{\dot{B}^{-1}_{\infty,\infty}}\|w\|_{L^{2}}\|\nabla w\|_{L^{2}}\nonumber\\
  &\leq \frac{\mu}{8}\|\nabla w\|_{L^{2}}^{2}+C\|w\|_{\dot{B}^{-1}_{\infty,\infty}}^{2}\|w\|_{L^{2}}^{2}.
\end{align}
The second term on the right-hand side of \eqref{eq5.3} can be
estimated the same as \eqref{eq4.6}:
\begin{equation}\label{eq5.5}
  \Big{|}-\int_{Q}\frac{\delta
  E}{\delta\phi}\nabla\phi\cdot\nabla\times w dx\Big{|}\leq
  \frac{\mu}{8}\|\nabla w\|_{L^{2}}^{2}+C\|\frac{\delta
  E}{\delta\phi}\|_{L^{2}}^{2}.
\end{equation}
Taking \eqref{eq5.4} and \eqref{eq5.5} into \eqref{eq5.3}, we obtain
\begin{equation}\label{eq5.6}
  \frac{d}{dt}\|w\|_{L^{2}}^{2}+\frac{3\mu}{2}\|\nabla w\|_{L^{2}}^{2}\leq
  C\Big(\|w\|_{\dot{B}^{-1}_{\infty,\infty}}^{2}+1\Big)\Big(\|w\|_{L^{2}}^{2}+\|\frac{\delta
  E}{\delta\phi}\|_{L^{2}}^{2}\Big).
\end{equation}
The estimate for $\frac{\delta
  E}{\delta\phi}$ can be proceeded the same as that in the proof of Theorem \ref{th1.1}, thus we also get \eqref{eq4.17}. Multiplying  \eqref{eq4.17} by $\eta$ and adding
\eqref{eq5.6} together, we obtain
\begin{align}\label{eq5.7}
  \frac{d}{dt}\Big(\|w\|_{L^{2}}^{2}&+\eta\|\frac{\delta
  E}{\delta\phi}\|_{L^{2}}^{2}\Big)+\mu\|\nabla w\|_{L^{2}}^{2}+k\varepsilon\gamma\eta\|\Delta\frac{\delta
  E}{\delta\phi}\|_{L^{2}}^{2}\nonumber\\
  &\leq
  C\Big(\|w\|_{\dot{B}^{-1}_{\infty,\infty}}^{2}+1\Big)\Big(\|w\|_{L^{2}}^{2}+\eta\|\frac{\delta
  E}{\delta\phi}\|_{L^{2}}^{2}+1\Big)\nonumber\\
  &\leq
  C\Big(\|\omega\|_{\dot{B}^{-1}_{\infty,\infty}}^{2}+1\Big)\Big(\|w\|_{L^{2}}^{2}+\eta\|\frac{\delta
  E}{\delta\phi}\|_{L^{2}}^{2}+1\Big)\nonumber\\
  &= C\frac{\|\omega\|_{\dot{B}^{-1}_{\infty,\infty}}^{2}+1}{1+\ln(e+\|\omega\|_{\dot{B}^{-1}_{\infty,\infty}})}
  \Big(\|w\|_{L^{2}}^{2}+\eta\|\frac{\delta
  E}{\delta\phi}\|_{L^{2}}^{2}+1\Big)\Big[1+\ln(e+\|\nabla\Delta u\|_{L^{2}})\Big]\nonumber\\
  &\leq C\frac{\|\omega\|_{\dot{B}^{-1}_{\infty,\infty}}^{2}+1}{1+\ln(e+\|\omega\|_{\dot{B}^{-1}_{\infty,\infty}})}\Big(\|w\|_{L^{2}}^{2}
  +\eta\|\frac{\delta
  E}{\delta\phi}\|_{L^{2}}^{2}+1\Big)\Big[1+\ln(e+\|\nabla\Delta u\|_{L^{2}})\Big],
\end{align}
where $C$ is a constant which may depend on $\eta$.

\medskip

By the condition \eqref{eq1.21}, one concludes that for any small
constant $\sigma>0$, there exists $T_{0}<T$ such that
\begin{equation}\label{eq5.8}
  \int_{T_{0}}^{T}\frac{\|\omega\|_{\dot{B}^{-1}_{\infty,\infty}}^{2}+1}{1+\ln(e+\|\omega\|_{\dot{B}^{-1}_{\infty,\infty}})}dt<\sigma.
\end{equation}
For any $T_{0}<t\leq T$, we set
\begin{equation}\label{eq5.9}
  h(t):=\sup_{T_{0}\leq \tau\leq t}\Big(\|\nabla\Delta
  u(\tau)\|_{L^{2}}^{2}+\hat{\eta}\|\Delta\frac{\delta E}{\delta
  \phi}(\tau)\|_{L^{2}}^{2}\Big),
\end{equation}
where $\hat{\eta}$ is a determined constant which specified later.
Applying Gronwall's inequality to \eqref{eq5.7} in the time interval
$[T_{0},t]$, one has
\begin{align}\label{eq5.10}
  \|\nabla u(t)\|_{L^{2}}^{2}&+\eta\|\frac{\delta
  E}{\delta\phi}(t)\|_{L^{2}}^{2}\nonumber\\
  &\leq C_{0}\exp\Big(C(1+\ln(e+h(t)))
  \int_{T_{0}}^{t}\frac{\|\omega\|_{\dot{B}^{-1}_{\infty,\infty}}^{2}+1}{1+\ln(e+\|\omega\|_{\dot{B}^{-1}_{\infty,\infty}})}d\tau\Big)\nonumber\\
  &\leq C_{0}\exp\big(C\sigma(1+\ln(e+h(t)))\big)\nonumber\\
  &\leq C_{0}(e+h(t))^{2C\sigma},
\end{align}
where $C_{0}=\|\nabla u(T_{0})\|_{L^{2}}^{2}+\eta\|\frac{\delta
E}{\delta\phi}(T_{0})\|_{L^{2}}^{2}$ is a positive constant
depending on $T_{0}$.

\medskip

The derivations of higher derivative estimates are analogously the
proof of Theorem \ref{th1.1}, thus we safely omit it. This completes
the proof of Theorem \ref{th1.2}.

\end{document}